\documentclass[a4paper]{article}
\usepackage{amsmath,amsthm,amssymb,mathrsfs}
\usepackage{enumerate}
\usepackage[dvipdfm]{pict2e}
\usepackage{bbm,cases}
\usepackage{color}
\usepackage{graphicx}

\newtheorem{dfn}{Definition}[section]
\newtheorem{thm}[dfn]{Theorem}
\newtheorem{lem}[dfn]{Lemma}
\newtheorem{cor}[dfn]{Corollary}
\newtheorem{prop}[dfn]{Proposition}\makeatletter

             \@addtoreset{equation}{section}
\makeatother

\newcommand{\y}{{\mathbbm y}}
\newcommand{\x}{{\mathbbm x}}

\newcommand{\Y}{\mathbb{Y}}

\newcommand{\dis}{\displaystyle}
\newcommand{\ve}{\varepsilon}
\begin{document}
\title{Branching random walks in random environment and super-Brownian motion in random environment}
\author{Makoto Nakashima \footnote{nakamako@math.tsukuba.ac.jp, Division of Mathematics, Graduate School of Pure and Applied Sciences,University of Tsukuba, 1-1-1 Ten-noudai, Tsukuba-shi, Ibaraki-ken, Japan } }
\date{}
\maketitle

\begin{abstract}
We focus on the existence and  characterization of the limit for a certain critical branching random walks in time-space random environment in one dimension which was introduced in \cite{BGK}. Each particle performs simple random walk on $\mathbb{Z}$ and branching mechanism depends on the time-space site. The weak limit of this measure valued processes is  characterized as a solution of the non-trivial martingale problem and called super-Brownian motions in random environment in \cite{Myt}. Moreover, we will show the weak uniqueness of the solutions with some initial condition.
\end{abstract}

\vspace{1em}

\vspace{1em}{\bf Key words:} Superprocesses in random environment, Branching random walks in random environment, Stochastic heat equations, Uniqueness

\vspace{1em}
We denote by  $(\Omega, {\cal F},P )$ a probability space. Let $\mathbb{N}=\{0,1,2,\cdots\}$, $\mathbb{N}^*=\{1,2,3,\cdots\}$, and $\mathbb{Z}=\{0,\pm 1,\pm 2,\cdots\}$.

\section{Introduction}\label{0}

Super-Brownian motion(SBM) is a measure valued process which was introduced by Dawson and Watanabe independently\cite{Daw,Wat} and is  obtained as the limit of (asymptotically) critical  branching Brownian motions (or branching random walks). There are many books for introduction of super-Brownian motion \cite{Daw1,Eth} and dealing with several aspects of it \cite{Dyn1,Dyn2,LeG,Per}. Also, super-Brownian motion appears in physics and population genetics.

An example of the construction is the following, where we always treat Euclidean space as the space, $\mathbb{R}^d$ in this paper.

We assume that at time $0$, there are $N$ particles in $\mathbb{Z}^d$ as the $0$-th generation particle. Each of $N$ particles chooses independently of each others a nearest neighbor site uniformly, moves there at time $1$, and then each particle independently of each others either dies or split into two particles with probability $1/2$ (1st generation). The newly produced particles in $n$-th generation  perform in the same manner, that is each of them chooses independently of each others a nearest neighbor site uniformly, moves there at time $n+1$, and then each particle independently of each others  either dies or split into 2 particles with probability $1/2$.

Let $X^{(N)}_{t}(\cdot)$ be the measure-valued Markov processes defined by \begin{align}
{\dis X_{t}^{(N)}(B)=\frac{\sharp \left\{			\text{particles in $B\sqrt{N}$ at  $\lfloor tN\rfloor$-th generation at time $tN$		}		\right\}	}{N}},\notag
\end{align}
where $B\in {\cal B}(\mathbb{R}^d)$ are Borel sets in $\mathbb{R}^{d}$ and $B\sqrt{N}=\{x=y\sqrt{N}\text{ for }y\in {B} \}$. Then, under some conditions, they converge as $N\to \infty$ to a measure-valued processes, \textit{super-Brownian motion}.  In particular, the limit $X_{t}(\phi)$ is characterized as the unique solution of the martingale problem:\begin{align}
\begin{cases}
\text{For all }\phi\in {\cal D}(\Delta),\\
{\dis \quad \quad Z_{t}(\phi):=X_{t}(\phi)-X_{0}(\phi)-\int_{0}^t \frac{1}{2d}X_{s}\left({\Delta \phi}\right)ds}\\
\text{is an ${\cal F}_{t}^X$-continuous square-integrable martingale}\\
{\dis \quad\quad Z_{0}(\phi)=0\ \text{ and } \left\langle Z(\phi)		\right\rangle_t=\int_{0}^tX_{s}(\phi^2)ds},
\end{cases}\label{SBM}
\end{align}
where $\nu(\phi)=\int \phi d\nu$ for any measure $\nu$.

It is a well-known fact  that one-dimensional super-Brownian motion is related to stochastic heat equation(\cite{KoShi,Rei}). When $d=1$, super-Brownian motion $X_t(dx)$ is almost surely absolutely continuous with respect to Lebesgue measure and its density $u(t,x)$ satisfies the following stochastic heat equation:\begin{align*}
\frac{\partial }{\partial t}u=\frac{1}{2}\Delta u+\sqrt{u}\dot{W}(t,x),
\end{align*}
where $\dot{W}(t,x)$ is space-time white noise. On the other hand, for $d\geq 2$, $X_t(\cdot)$ is almost singular with respect to Lebesgue measure.(\cite{DawPer,LeGPerTay,Per1,Per2})

In this paper, we consider super-Brownian motion in random environment, which are introduced in \cite{Myt}. Mytnik showed the existence and uniqueness of the scaling limit $X_{t}(\cdot)$ for a certain critical branching diffusion in random environment with some conditions. It is characterized as the unique solution of the martingale problem: \begin{align}
\begin{cases}
\text{For all }\phi\in {\cal D}(\Delta),\\
{\dis \quad \quad Z_{t}(\phi):=X_{t}(\phi)-X_{0}(\phi)-\int_{0}^t \frac{1}{2}X_{s}\left({\Delta \phi}\right)ds}\\
\text{is an ${\cal F}_{t}^X$-continuous square-integrable martingale and}\\
{\dis \quad\quad \left\langle Z(\phi)		\right\rangle_t=\int_{0}^tX_{s}(\phi^2)ds+\int_{0}^t	\int_{\mathbb{R}^d\times \mathbb{R}^d}g(x,y)\phi(x)\phi(y)X_{s}(dx)X_{s}(dy)ds}	,
\end{cases}\label{mytsbm}
\end{align}
where  $g(\cdot,\cdot)$ is bounded continuous function in a certain class. 
In this paper, we construct a super-Brownian motion in random environment as a limit point of scaled branching random walks in random environment, which is a solution of (\ref{mytsbm}) for the case where $g(x,y)$ is replaced by $\delta_{x,y}$. The definition of such martingale problem is formal. The rigorous definition will be given later.



\section{Branching random walks in random environment}
Before giving the system of the branching random walks in random environment, we introduce Ulam-Harris tree ${\cal T}$ for labeling the particles. We set $T_{k}=\left(\mathbb{N}^*\right)^{k+1}$ for $k\geq 1$. Then, Ulam-Harris tree ${\cal T}$ is defined by $\displaystyle {\cal T}=\bigcup_{k\geq 0}T_{k}$.

We will give a name to each particle by using elements of ${\cal T}$.\begin{enumerate}[i)]
\item When there are $M$ particles at the $0$-th generation, we label them as $1,2,\cdots,M\in T_{0}$.
\item If the $n$-th generation particle $\x=(x_{0},\cdots,x_{n})\in T_{n}$ gives birth to $k_{\x}$ particles, then we name them as $(x_{0},\cdots,x_{n},1),\cdots,(x_{0},\cdots,x_{n},k_{\x})\in T_{n+1}$.
\end{enumerate}
Thus, every particle in the branching systems has its own name in ${\cal T}$. We define  $|\x|$ by its generation, that is if $\x$ is an element of $T_{k}$, then $|\x|=k$. For convenience, we denote by $|\x\wedge \y|$ the generation of the closest common ancestor of $\x$ and $\y$. If $\x$ and $\y$ have no common ancestor, then we define $|\x\wedge \y|=-\infty$.  Also, we denote by $\y/\x$ the last digit of $\y$ when $\y$ is a child of $\x$, that is  \begin{align*}
\y/\x=\begin{cases}
k_{\y},\ \ \text{if }\begin{subarray}{l}				
\x=(x_0,\cdots,x_n)\in T_n,\\
\y=(x_0,\cdots,x_n,k_{\y})\in T_{n+1},
\end{subarray}\ \  \text{for some $n\in \mathbb{N}$,}
\\
\infty,\ \ \ \ \text{otherwise.}
\end{cases}
\end{align*}

Now, we give the definition of branching random walks in random environment. In our case, particle move on $\mathbb{Z}$ and the process evolves by the following rule:
\begin{enumerate}[i)]
\item The initial particles locate at site $\{x_i\in 2\mathbb{Z}:i=1,\cdots,M_N\}$.
\item Each particle located at site $x$ at  time $n$ chooses a nearest neighbor site  independently of each others  with probability $\frac{1}{2}$ and moves there at time $n+1$. Simultaneously, it is replaced by $k$-children with probability $q^{(N)}_{n,x}(k)$
 independently of each others,
 \end{enumerate}
where $\left\{\left\{q^{(N)}_{n,x}(k)\right\}_{k=0}^{\infty}:(n,x)\in \mathbb{N}\times \mathbb{Z}\right\}$ are the offspring distributions assigned to each time-space site $(n,x)$ which are i.i.d.\,in $(n,x)$.
We denote by $B_{n}^{(N)}$ and by $B^{(N)}_{n,x}$ the total number of particles at time $n$ and the local number of particles at site $x$ at time $n$. Also, we denote by $m^{(N,p)}_{n,x}$ the $p$-th moment of offsprings for offspring distribution $\{q_{n,x}^{(N)}(k)\}$, that is \begin{align}
m_{n,x}^{(N,p)}=\sum_{k=0}^{\infty}k^{p}q_{n,x}^{(N)}(k).\notag
\end{align}

This model is called branching random walks in random  environment (BRWRE) whose properties as measure valued processes is for ``supercritical" case are studied well \cite{HN,HNY}. Also, the continuous counterpart, branching Brownian motions in random environment is introduced by Shiozawa\cite{Shi1,Shi2}. We know that  the normalized random measure weakly converges to Gaussian measure in probability in one phase, whereas the localization has occurred in the other phase.

In this paper,  we focus on the scaled measure valued processes $X_{t}^{(N)}$ associated to this branching random walks:\begin{align}
&X^{(N)}_{0}=\frac{1}{N}\sum_{i=0}^{M_N}\delta_{{x_{i}}/{N^{\frac{1}{2}}}},\notag
\intertext{and }
&X_{t}^{(N)}=\frac{1}{N}\sum_{i=1}^{B^{(N)}_{tN}}\delta_{{x_{i}(t)}/{N^{\frac{1}{2}}}},\ \ \text{for $t=\frac{1}{N},\cdots, \frac{\lfloor KN\rfloor}{N}$ for each $K>0$,}\notag
\end{align}
 where $x_{i}(t)$ is the position of the $i$-th particle at $tN$-th generation.
We remark that if we identify $B_{tN,x}^{(N)}$ as the measure $B^{(N)}_{tN,x}\delta_{x}$, then $X_{t}^{(N)}$ is represented as 
\begin{align}
X_{t}^{(N)}=\frac{1}{N}\sum_{x\in \mathbb{Z}}B^{(N)}_{tN,x}\delta_{x/N^{\frac{1}{2}}} \ \ \ \text{for $t=\frac{1}{N},\cdots,\frac{\lfloor KN\rfloor}{N}$.}\notag
\end{align}
Let  ${\cal M}_{F}(\mathbb{R})$ be the set of the finite Borel measures on $\mathbb{R}$. 
For convenience, we extend this model to the c{\'a}dl{\'a}g paths in ${\cal M}_{F}(\mathbb{R})$ by
\begin{align}
X_{t}^{(N)}=\frac{1}{N}\sum_{x\in \mathbb{Z}}B^{(N)}_{\underline{t}N,x}\delta_{x/N^{\frac{1}{2}}},\ \ \text{for }\underline{t}\leq t<\underline{t}+\frac{1}{N},\notag
\end{align}
where we define $\underline{t}$ for $t$ and $N$ by some positive number $\frac{i}{N}$ for $i\in \mathbb{N}$ satisfying $\frac{i}{N}\leq t< \frac{i+1}{N}$. 			Then,  $X_{t}^{(N)}\in {\cal M}_{F}(\mathbb{R})$ for each $t\in[0,K]$. Let $\phi\in {\cal B}_{b}(\mathbb{R})$, where ${\cal B}_{b}(\mathbb{R})$ is the set of the bounded Borel measurable functions on $\mathbb{R}$. We denote the product of $\nu\in{\cal M}_{F}(\mathbb{R})$ and $\phi\in {\cal B}_{b}(\mathbb{R})$ by $\nu(\phi)$, that is  \begin{align}
\nu(\phi)=\int_{\mathbb{R}}\phi(x)\nu(dx).\notag
\end{align}


To describe the main theorem, we give the following assumption on the environment:

{\bf Assumption A} \begin{align*}
&E[m^{(N,1)}_{0,0}]=E\left[\sum_{i=0}^{\infty}k\,q_{n,x}^{(N)}\right]=1,\lim_{N\to \infty}E\left[m^{(N,2)}_{0,0}-1\right]=\gamma>0,\\
&\sup_{N\geq 1}E\left[m^{(N,4)}_{0,0}\right]<\infty, \lim_{N\to \infty }N^{\frac{1}{2}}E\left[(m_{0,0}^{(N,1)}-1)^2\right]=\beta^2,\\
&\sup_{N\geq 1}N^{\frac{1}{2}}E\left[(m^{(N,1)}_{0,0}-1)^4\right]<\infty.
\end{align*}

{\bf Example:} The simplest example satisfying Assumption A is the case where $\dis q_{n,x}^{(N)}(0)=\frac{1}{2}-\frac{\beta \xi(n,x)}{2N^{\frac{1}{4}}}, q_{n,x}^{(N)}(2)=\frac{1}{2}+\frac{\beta \xi(n,x)}{2N^{\frac{1}{4}}}$ for i.i.d.\,random variables $\{\xi(n,x):(n,x)\in \mathbb{N}\times \mathbb{Z}\}$ such that $P(\xi(n,x)=1)=P(\xi(n,x)=-1)=\frac{1}{2}$.

Before giving our main theorem, we introduce a set of functions on $\mathbb{R}$, rapidly decreasing continuous functions:
\begin{align*}
C_{rap}(\mathbb{R})=\left\{	g\in C_b(\mathbb{R}): |g|_p= \sup_x e^{p|x|}g(x)<\infty,\ \text{for all }p>0			\right\}.
\end{align*}

\begin{thm}\label{thm1}
We suppose that $X_{0}^{(N)}(\cdot)\Rightarrow X_0(\cdot)$ in ${\cal M}_F(\mathbb{R})$ and Assumption A.
Then, the sequence of measure valued processes $\left\{X_{\cdot}^{(N)}:N\in\mathbb{N}\right\}$ converges to  a continuous measure valued process $X_{\cdot}\in  C([0,\infty),{\cal M}_{F}(\mathbb{R}))$. Moreover, for any $t>0$, any limit point $X_t(dx)$ is almost surely absolutely continuous with respect to Lebesgue measure and its density $u(t,x)$ is a solution of the  following martingale problem: \begin{align}
\begin{cases}
\text{For all }\phi\in {\cal D}(\Delta),\\[.5em]
{\dis  \qquad Z_{t}(\phi)=\int_{\mathbb{R}}\phi(x)u(t,x)dx-\int_{\mathbb{R}}\phi(x)X_0(dx)-\frac{1}{2}\int_{0}^{t}\int_{\mathbb{R}}\Delta \phi(x)u(t,x)dxds}\\[.5em]
\text{is an ${\cal F}_{t}^{X}$-continuous square-integrable martingale and }\\
{\dis \qquad \qquad \qquad \langle Z(\phi)\rangle_{t}=\int_{0}^{t}\int_{\mathbb{R}}\phi^{2}(x)\left(\gamma u(s,x)+{2\beta^2}u(s,x)^2\right)dxds}.\\
\end{cases}\label{mart1}
\end{align}
In particular, if $X_0$ has a density $u\in C_{rap}^+(\mathbb{R})$, then the solutions to (\ref{mart1}) is the unique.
\end{thm}


{\bf Remark:} We found from Assumption A that the fluctuation of the environment is mainly given by $(m_{n,x}^{(N,1)}-1)$ and scaling factor is $N^{-\frac{1}{4}}$. (It appears clearly in the Example beyond Assumption A.) This scaling factor is different from $N^{-\frac{1}{2}}$, the one in \cite{Myt}. When the scaling factor is $N^{-\frac{1}{2}}$, the limit is the usual super-Brownian motion (\ref{SBM}).

We roughly discuss how the scaling factor in our model is determined. For simplicity, we consider the model for the case where the environment is the one given in Example. 

We scale the space by $N^{-\frac{1}{2}}$. Then, the summation of the fluctuation of the first moment of offsprings in the segment $\{k\}\times [x,y]$ is $\dis \sum_{z\in [xN^{{1}/{2}},yN^{{1}/{2}}]}\frac{\beta \xi(k,z)}{N^{\frac{1}{4}}}			$. Since it is the summation of i.i.d.\,random variables of $\frac{(y-x)N^{\frac{1}{2}}}{2}$, the central limit theorem holds and it weakly converges to a Gaussian random variable with distribution $N(0,\frac{\beta^2(y-x)}{2})$.  Similar argument holds for random variables other than Bernoulli random variables.

{\bf Remark:} 
The martingale problem (\ref{mart1}) is the rigorous and general definition of the martingale problem when $g(x,y)$ is replaced by $\delta_{x-y}$ in (\ref{mytsbm}). 
Also, the theorem implies the existence of the solution to the stochastic heat equation\begin{align}
\frac{\partial }{\partial t}u=\frac{1}{2}\Delta u+\sqrt{\gamma u+2{\beta^2}u^{2}}\dot{W},\label{spde}
\end{align}
and $\lim_{t\to+0}u(t,x)dx=X_0(dx)$, 
where $\dot{W}$ is time-space white noise. 
In \cite{MP}, the existence of solutions for general SPDE containing (\ref{spde}) when the initial measure $X_0(dx)$ has  a continuous density with rapidly decreasing at infinity. 

Also, Theorem \ref{thm1} states that  the uniqueness in law of solutions to (\ref{spde}). 
There are a lot of papers on uniqueness of the stochastic heat equation $\frac{\partial }{\partial t}u=\frac{1}{2}\Delta u+|u|^{\gamma}\dot{W}$. It is known that weak uniqueness holds for $\frac{1}{2}\leq \gamma\leq 1$ in \cite{Myt2} and pathwise  uniqueness holds for $\frac{3}{4}<\gamma\leq 1$ in \cite{MytPer}. However, pathwise uniqueness fails when solutions are allowed to take negative values for $\gamma <\frac{3}{4}$ in \cite{MMP}. 

Especially, we should remark that in the proof of Theorem \ref{thm1}, we will show that the weak uniqueness for the limit points of $\{X^{(N)}\}$ but it will not imply the weak uniqueness for the solutions to the martingale problem (\ref{mart1}). It is because we will use some estimate arising from $\{X^{(N)}\}$.

\section{Proof of Theorem \ref{thm1}}\label{3}
In this section, we will give a proof of Theorem \ref{thm1}. The proof is divided into three steps:\begin{enumerate}[i)]
\item Tightness.
\item Identification of the limit point process.
\item Weak uniqueness of the limit points.
\end{enumerate}

In this section, we consider the following setting for simplicity. \vspace{.5em}

{\bf Assumption B:}  The number of initial particles is $N$ and all of them locates at the origin at time $0$. Also, $\dis q_{n,x}^{(N)}(0)=\frac{1}{2}-\frac{\beta \xi(n,x)}{2N^{\frac{1}{4}}}, q_{n,x}^{(N)}(2)=\frac{1}{2}+\frac{\beta \xi(n,x)}{2N^{\frac{1}{4}}}$ for i.i.d.\,random variables $\{\xi(n,x):(n,x)\in \mathbb{N}\times \mathbb{Z}\}$ such that $P(\xi(n,x)=1)=P(\xi(n,x)=-1)=\frac{1}{2}$.\vspace{.5em}

To consider the general model, it is almost enough to replace $\frac{\beta \xi(n,x)}{N^{\frac{1}{4}}}$ by $m_{n,x}^{(N,1)}-1$. We sometimes need to consider $\{\{q^{(N)}_{n,x}{(k)}\}_{k\geq 0}:(n,x)\in\mathbb{N}\times \mathbb{Z}\}$. Especially, $\gamma$ appears in the same situation as the construction of the usual super-Brownian motion, so the reader will not to have any difficulties.

\vspace{1em}
Before staring the proof, we will look at the $X_{t}^{(N)}(\phi)$.
Since $X^{(N)}_{t}$ are constant in $t\in[\underline{t},\underline{t}+\frac{1}{N})$, it is enough to see the difference between $X_{\underline{t}}^{(N)}$ and $X_{\underline{t}+\frac{1}{N}}^{(N)}$; \begin{align}
X_{\underline{t}+\frac{1}{N}}^{(N)}\left(\phi \right)-X_{\underline{t}}^{(N)}(\phi)&=\frac{1}{N}\sum_{\x\sim \underline{t}}\left(\phi\left(	\frac{Y^{\x}_{\underline{t}N+1}}{N^{\frac{1}{2}}}			\right)V^{\x}-	\phi\left(	\frac{Y^{\x}_{\underline{t}N}}{N^{\frac{1}{2}}}		\right)	\right),	\notag
\end{align}
where $\x\sim \underline{t}$ means that the particle $\x$ is the $\underline{t}N$-th generation, $Y_{\underline{t}N}^{\x}$ is the position of the particle $\x$ at time $\underline{t}N$ for $\x\sim \underline{t}$, $V^{\x}$ is the number of children of $\x$ and for simplicity, we omit $N$. We remark that $Y^{\x}_{\underline{t}N+1}=Y^{\y}_{\underline{t}N+1}$ for $\y$ which is  a child of $\x$.

Also, we divide this summation into four parts:
\begin{align}
&(LHS)\notag\\
&=\frac{1}{N}\sum_{\x\sim\underline{t}}\phi\left(\frac{Y_{\underline{t}N+1}^{\x}}{N^{\frac{1}{2}}}\right)\left(V^{\x}-1-\frac{\beta\xi(\underline{t}N,Y_{\underline{t}N}^{\x})}{N^{\frac{1}{4}}}\right)\notag\\
&\ \ +\frac{1}{N}\sum_{\x\sim \underline{t}}\phi\left(\frac{Y_{\underline{t}N+1}^{\x}}{N^{\frac{1}{2}}}\right)\frac{\beta\xi\left(\underline{t}N,Y_{\underline{t}N}^{\x}\right)}{N^{\frac{1}{4}}}\notag\\
&\ \ +\frac{1}{N}\sum_{\x\sim \underline{t}}				\left(	\phi\left(	\frac{Y_{\underline{t}N+1}^{\x}}{N^{\frac{1}{2}}}			\right)-\phi\left(		\frac{Y^{\x}_{\underline{t}N}	}{N^{\frac{1}{2}}}\right)	-\frac{\phi\left(		\frac{Y^{\x}_{\underline{t}N}+1	}{N^{\frac{1}{2}}}\right)+\phi\left(		\tfrac{Y^{\x}_{\underline{t}N}	-1}{N^{\frac{1}{2}}}\right)-2\phi\left(		\tfrac{Y^{\x}_{\underline{t}N}	}{N^{\frac{1}{2}}}\right)}{2}
\right)
\notag\\
& \ \ +\frac{1}{N}\sum_{\x\sim \underline{t}}\frac{\phi\left(		\frac{Y^{\x}_{\underline{t}N}+1	}{N^{\frac{1}{2}}}\right)+\phi\left(		\frac{Y^{\x}_{\underline{t}N}-1	}{N^{\frac{1}{2}}}\right)-2\phi\left(		\frac{Y^{\x}_{\underline{t}N}	}{N^{\frac{1}{2}}}\right)}{2}\notag\\
&=\Delta M_{\underline{t}}^{(b,N)}(\phi)+\Delta M_{\underline{t}}^{(e,N)}(\phi)+\Delta M_{\underline{t}}^{(s,N)}(\phi)\notag\\
&\hspace{5em}+\frac{1}{N}\sum_{\x\sim \underline{t}}\frac{\phi\left(		\frac{Y^{\x}_{\underline{t}N}+1	}{N^{\frac{1}{2}}}\right)+\phi\left(		\frac{Y^{\x}_{\underline{t}N}-1	}{N^{\frac{1}{2}}}\right)-2\phi\left(		\frac{Y^{\x}_{\underline{t}N}	}{N^{\frac{1}{2}}}\right)}{2}.\notag
\end{align}


Thus, we have that for $0\leq\underline{t} \leq t< \underline{t}+\frac{1}{N}$\begin{align}
X_{t}^{(N)}(\phi)-X_{0}^{(N)}(\phi)=\left(M_{{t}}^{(b,N)}(\phi)+M_{{t}}^{(e,N)}(\phi)+M_{{t}}^{(s,N)}(\phi)\right)+\int_{0}^{\underline{t}}X_{s}^{(N)}\left(A^{N}\phi\right)ds,\label{martiN}
\end{align}
where \begin{align}
&M_{{t}}^{(b,N)}(\phi)=\frac{1}{N}\sum_{\underline{s}<t}\sum_{\x\sim\underline{s}}\phi\left(\frac{Y_{\underline{s}N+1}^{\x}}{N^{\frac{1}{2}}}\right)\left(V^{\x}-1-\frac{\beta\xi(\underline{s}N,Y_{\underline{s}N}^{\x})}{N^{\frac{1}{4}}}\right),\notag\\
&M_{{t}}^{(e,N)}(\phi)=\frac{1}{N}\sum_{\underline{s}<t}\sum_{\x\sim \underline{s}}\phi\left(\frac{Y_{\underline{s}N+1}^{\x}}{N^{\frac{1}{2}}}\right)\frac{\beta\xi\left(\underline{s}N,Y_{\underline{s}N}^{\x}\right)}{N^{\frac{1}{4}}},\notag\\
&M_{{t}}^{(s,N)}(\phi)=\frac{1}{N}\sum_{\underline{s}<t}\sum_{\x\sim \underline{s}}				\left(	\phi\left(	\frac{Y_{\underline{s}N+1}^{\x}}{N^{\frac{1}{2}}}			\right)-\phi\left(		\frac{Y^{\x}_{\underline{s}N}	}{N^{\frac{1}{2}}}\right)	-\frac{\phi\left(		\frac{Y^{\x}_{\underline{s}N}+1	}{N^{\frac{1}{2}}}\right)+\phi\left(		\frac{Y^{\x}_{\underline{s}N}-1	}{N^{\frac{1}{2}}}\right)-2\phi\left(		\frac{Y^{\x}_{\underline{s}N}	}{N^{\frac{1}{2}}}\right)}{2}
\right),\notag
\end{align}
and $A^{N}: {\cal B}_{b}(\mathbb{R})\to {\cal B}_{b}(\mathbb{R})$ is the following operator;\begin{align}
A^{N}\phi(x)=\frac{\phi\left(x+\frac{1}{N^{\frac{1}{2}}}\right)+\phi\left(x-\frac{1}{N^{\frac{1}{2}}}\right)-2\phi(x)}{\frac{2}{N}}.\notag
\end{align}
Actually, we have that \begin{align}
\int_{0}^{\underline{t}}X_{s}^{(N)}\left(A^{N}\phi\right)ds=\sum_{\underline{s}<t}\sum_{\x\sim \underline{s}}
\frac{1}{N}A^{N}\phi\left(	\frac{	Y_{\underline{s}N}^{\x}}{N^{\frac{1}{2}}}	\right).\notag
\end{align}

Also, we remark that $M_{\underline{t}}^{(b,N)}(\phi)$, $M_{\underline{t}}^{(e,N)}(\phi)$, and $M_{\underline{t}}^{(s,N)}(\phi)$ are ${\cal F}^{(N)}_{\underline{t}N}$-martingales, where ${\cal F}^{(N)}_{n}$ is the $\sigma$-algebra \begin{align}
\sigma \left(	V^{\x},Y^{\x}_{k+1},\xi(k,x):	|\x|\leq n-1,	k\leq n-1,x\in \mathbb{Z}	\right)\notag,
\end{align}
where ${\cal F}_{0}^{(N)}=\{\emptyset, \Omega\}$.
Indeed, since $Y^{\x}_{n+1}$ are independent of  $V^{\x}$ and $\xi(n,x)$,
\begin{align}
&E\left[\left.M^{(b,N)}_{\underline{t}}(\phi)-M^{(b,N)}_{\underline{t}-\frac{1}{N}}(\phi)\right|	{\cal F}_{\underline{t}N-1}^{(N)}	\right]\notag\\
&\hspace{3.5em}=\frac{1}{N}\sum_{\x\sim\underline{t}-\frac{1}{N}}E\left[\left.	\phi\left(		\frac{Y^{\x}_{\underline{t}N}}{N^{\frac{1}{2}}}	\right)						\right|{\cal F}_{\underline{t}N-1}^{(N)}\right]E\left[	\left.V^{\x}-1-\frac{\beta\xi\left(		\underline{t}N-1,Y^{\x}_{\underline{t}N-1}		\right)}{N^{\frac{1}{4}}}\right|{\cal F}_{\underline{t}N-1}^{(N)}			\right]\notag\\
&\hspace{3.5em}=0,\notag\\
&E\left[	\left.			M^{(e,N)}_{\underline{t}}(\phi)-M^{(e,N)}_{\underline{t}-\frac{1}{N}}(\phi)	\right|	{\cal F}_{\underline{t}N-1}^{(N)}			\right]\notag\\
&\hspace{3.5em}=\frac{1}{N}\sum_{\x\sim \underline{t}-\frac{1}{N}}E\left[\left.			\phi\left(	\frac{Y^{\x}_{\underline{t}N}}{N^{\frac{1}{2}}}		\right)	\right|{\cal F}_{\underline{t}N-1}^{(N)}\right]E\left[\left.		\frac{\beta\xi\left(\underline{t}N-1,Y^{\x}_{\underline{t}N-1}\right)}{N^{\frac{1}{4}}}		\right|{\cal F}_{\underline{t}N-1}^{(N)}\right]\notag\\
&\hspace{3.5em}=0,\notag\\
\intertext{and}
&E\left[\left.	M^{(s,N)}_{\underline{t}}(\phi)	-M^{(s,N)}_{\underline{t}-\frac{1}{N}}(\phi)	\right|{\cal F}_{\underline{t}N-1}^{(N)}\right]=0\notag,
\end{align}
almost surely.

Moreover, the decomposition (\ref{martiN}) is very useful since the martingales $M_{\underline{t}}^{(i,N)}(\phi)$ $i=b,e,s$ are orthogonal to each others. Indeed, we have that
\begin{align*}
&E\left[\left.\left(\Delta M_{\underline{t}}^{(b,N)}(\phi)\right)\left(\Delta M_{\underline{t}}^{(e,N)}(\phi)\right)\right|		{\cal F}_{\underline{t}N-1}^{(N)}	\right]\\
&\quad =\frac{1}{N^2}\sum_{\x,\x'\sim \underline{t}-\frac{1}{N}}\left(E\left[\left.\phi\left(		\frac{Y^{\x}_{\underline{t}N}}{N^{\frac{1}{2}}}		\right)\phi\left(		\frac{Y^{\x'}_{\underline{t}N}}{N^{\frac{1}{2}}}		\right)\right| {\cal F}_{\underline{t}N-1}^{(N)}\right]	\right.\\
&\quad \times	\left.E\left[	E\left[	\left.\left.	\left(V^{\x}-1-\frac{\beta\xi(\underline{t}N-1,Y^{\x}_{\underline{t}N-1})}{N^{\frac{1}{4}}}\right)\right|{\cal G}_{\underline{t}N-1}^{(N)}\right]		\frac{\beta\xi(\underline{t}N-1,Y^{\x}_{\underline{t}N-1})}{N^{\frac{1}{4}}}\right|{\cal F}_{\underline{t}N-1}^{(N)}\right]\right)\\
&\quad =0,
\end{align*}
where $	{\cal G}^{(N)}_{n}={\cal F}_{n}^{(N)}\vee \sigma (\xi(n,x):x\in\mathbb{Z})	$ almost surely. 
Also, we can obtain by the similar argument that $E\left[\left.\left(\Delta M_{\underline{t}}^{(b,N)}(\phi)\right)\left(\Delta M_{\underline{t}}^{(s,N)}(\phi)\right)\right|		{\cal F}_{\underline{t}N-1}^{(N)}	\right]=E\left[\left.\left(\Delta M_{\underline{t}}^{(s,N)}(\phi)\right)\left(\Delta M_{\underline{t}}^{(e,N)}(\phi)\right)\right|		{\cal F}_{\underline{t}N-1}^{(N)}	\right]=0$ almost surely.


\subsection{Tightness}\label{tight}
In this subsection, we will prove the following lemma.
\begin{lem}\label{lem0}
The sequence $\{X^{(N)}\}$ is tight in $D([0,\infty),{\cal M}_{F}(\mathbb{R}))$, and each limit process is continuous.
\end{lem}

To prove it,  we will use the following theorem which reduces the problem to the tightness of real-valued process \cite[Theorem II.\,4.\,1]{Per}.

\begin{thm}\label{thm2}
Assume that $E$ is a Polish space. Let $D_{0}$ be a separating class of $C_{b}(E)$ containing $1$. A sequence of c{\'a}dl{\'a}g ${\cal M}_{F}(E)$-valued processes $\left\{X^{(N)}:N\in \mathbb{N}\right\}$ is $C$-relatively compact in $D\left([0,\infty), {\cal M}_{F}(E)\right)$ if and only if \begin{enumerate}[(i)]
\item for every $\varepsilon,T>0$, there is a compact set $K_{T,\varepsilon}$ in $E$ such that \begin{align}
\sup_{N}P\left(	\sup_{t\leq T}X_{t}^{(N)}\left(K_{T,\varepsilon}^{c}\right)>\varepsilon		\right)<\varepsilon, \notag
\end{align}
\item and for all $\phi\in D_{0}$, $\left\{X^{(N)}(\phi):N\in \mathbb{N}\right\}$ is $C$-relatively compact in $D\left([0,\infty),\mathbb{R}\right)$.
\end{enumerate}
\end{thm}

{\bf Assumption:} 
We choose $C_{b}^2(\mathbb{R})$ as $D_{0}$, where  $C_{b}^2(\mathbb{R})$ is the set of bounded continuous function on $\mathbb{R}$ with bounded derivatives of order $1$ and $2$.

\vspace{1em}
Hereafter, we will check the conditions (i) and (ii) of Theorem \ref{thm2} for our case. In the beginning, we give the proof of (ii) by using  the following lemmas:
\begin{lem}\label{lem1} 
For $\phi\in C_b^2(\mathbb{R})$, $\displaystyle \sup_{t\leq K}\left| M_{t}^{(s,N)}(\phi)\right|\stackrel{L^{2}}{\to}0$ as $N\to \infty$ for all $K>0$.
\end{lem}

\begin{lem}\label{lem2} (See \cite[Lemma II 4.5]{Per}.)
Let $\dis \left(M_{\underline{t}}^{(N)},	\overline{{\cal F}}_{\underline{t}}^{N}		\right)$ be discrete time martingales with $\dis M_{0}^{(N)}=0$. \\
Let $\dis \langle M^{(N)}\rangle_{\underline{t}}
=\sum_{0\leq \underline{s}<\underline{t}}E\left[	\left.\left(M_{\underline{s}+1/N}^{(N)}-	M_{\underline{s}}^{(N)}\right)^{2}\right|			\overline{{\cal F}}_{\underline{s}}^{N}	\right]
$, and we extend $\dis M_{\cdot}^{(N)}$ and $\dis\langle	M^{(N)}	\rangle_{\cdot}$ to $[0,\infty)$ as right continuous step functions.
\begin{enumerate}[(i)]
\item If $\left\{	\langle		M^{(N)}	\rangle_{\cdot}:N\in \mathbb{N}			\right\}$ is $C$-relatively compact in $D([0,\infty),\mathbb{R})$ and \begin{align}
\sup_{0\leq \underline{t}\leq K}\left|		M^{(N)}_{\underline{t}+1/N}-M^{(N)}_{\underline{t}}			\right|\stackrel{P}{\to}0 \ \ \text{as }N\to \infty \ \quad \text{for all }K>0, \label{Pconv}
\end{align}
then $\dis M^{(N)}_{\cdot}$ is $C$-relatively compact in $D([0,\infty),\mathbb{R})$.
\end{enumerate}
\item If, in addition, \begin{align}
\left\{	\left(M_{\underline{t}}^{(N)}\right)^{2}+\langle	M^{(N)}		\rangle_{\underline{t}}:N\in \mathbb{N}			\right\} \qquad \text{is uniformly integrable for all }\underline{t},\notag
\end{align}
then $M^{(N_{k})}_{\cdot}\stackrel{w}{\Rightarrow} M_{\cdot}$ in $D([0,\infty),\mathbb{R})$ implies $M$ is a continuous $L^{2}$-martingale and $\dis \left(M^{(N_{k})}_{\cdot},\langle M^{(N_{k})}\rangle_{\cdot} \right)\stackrel{w}{\Rightarrow}\left(		M_{\cdot},\langle  M\rangle_{\cdot}	\right)$ in $D([0,\infty),\mathbb{R})$.
\end{lem}


\begin{lem}\label{lem3} For any $\phi\in C_b^2(\mathbb{R})$,
the sequence $\dis C_t^{(N)}\left(\phi\right)\equiv \int_{0}^{\underline{t}}X_{s}^{(N)}\left(A^{N}\phi\right)ds$ is $C$-relatively compact in $D([0,\infty),\mathbb{R}).$

\end{lem}

When we can verify the conditions of Lemma \ref{lem2} for $M^{(b,N)}_{\cdot}(\phi)$, and  $M_{\cdot}^{(e,N)}(\phi)$, the sequence $\dis \left\{		X^{(N)}_{\cdot}(\phi):N\in \mathbb{N}	\right\}$ is $C$-relatively compact in $D([0,\infty),\mathbb{R})$. Moreover, if we check the condition of (i) in Theorem \ref{thm2}, then the tightness of $\left\{X^{(N)}_{\cdot}:N\in \mathbb{N}\right\}$ follows immediately.

\vspace{.8em}
Before starting the proof of the above lemmas, we prepare the following lemma. It tells us the mean of the measure $X_{\underline{t}}^{(N)}$ is the same as the distribution of the scaled simple random walk.

\begin{lem}\label{lem4}
We define \textit{historical process} by\begin{align}
H_{t}^{(N)}=\frac{1}{N}\sum_{\x\sim \underline{t}}\delta_{\frac{Y^{\x}_{(\cdot\wedge t)N}}{N^{1/2}}}\ \ \ \in {\cal M}_{F}\left(	D([0,\infty),\mathbb{R})			\right),\notag
\end{align}
where $Y^{\x}_{s}=Y^{\y}_{s}$ for $0\leq s< {|\x\wedge \y|+1}{}$, that is $Y^{\x}_{s}$ is the position of the $\left\lfloor sN\right\rfloor$-generation's ancestor of $\x$.

If $\psi:D([0,\infty),\mathbb{R})\to \mathbb{R}_{\geq 0}$ is Borel, then for any $t\geq 0$\begin{align}
E\left[		H_{t}^{(N)}(\psi)	\right]= E_Y\left[		\psi\left(	\frac{Y_{(\cdot \wedge t)N}}{N^{\frac{1}{2}}}	\right)		\right],\label{IP}
\end{align}
where $Y_{\cdot}$ is the trajectory of simple random walk on $\mathbb{Z}$.
In particular, for all $\phi\in {\cal B}_{+}(\mathbb{R}), $\begin{align}
E\left[	X_{\underline{t}}^{(N)}\left(\phi\right)		\right]=E_Y\left[\phi\left(\frac{Y_{\underline{t}N}}{N^{\frac{1}{2}}}\right)\right].\label{srw}
\end{align}
Moreover, for all $x$, $K>0$, we have that \begin{align}
P_Y\left(		\sup_{t\leq K }X_{t}^{(N)}(1)\geq x		\right)\leq x^{-1}.\label{bdd}
\end{align} 
\end{lem}

To prove this lemma, we introduce the notation. For $x({\cdot}),y(\cdot)\in D([0,\infty),\mathbb{R})$ such that $y(0)=0$, 
\begin{align*}
(x/s/y)(t)=\begin{cases}
x(t)			&\,\,\text{if }0\leq t<t,\\
x(s)+y(t-s)&\,\,\text{if }t\geq s. 
\end{cases}
\end{align*}
Then, $(x/s/y)(\cdot)\in  D([0,\infty),\mathbb{R})$. 

\begin{proof}

(\ref{IP}) follows from the Markov property. Indeed, 
 we have \begin{align}
E\left[	H_{\underline{t}}^{(N)}(\psi)			\right]&=	E\left[\frac{1}{N}	\sum_{\y\sim\underline{t}}	\psi\left(	\frac{Y^{\y}_{(\cdot\wedge \underline{t})N}	}{N^{\frac{1}{2}}}	\right)	\right]		\notag\\
&=E\left[		\frac{1}{N}\sum_{\x\sim \underline{t}-\frac{1}{N}}\psi\left(	\frac{Y^{\x}_{(\cdot\wedge \underline{t})N}}{N^{\frac{1}{2}}}		\right)E\left[\left.V^{\x}		\right|{\cal F}^{(N)}_{\underline{t}N-1}		\right]		\right]\notag\\
&=E\left[	E_{Z_1}\left[	\frac{1}{N}\sum_{\x\sim \underline{t}-\frac{1}{N}}	 \psi\left(		\frac{\left(Y^{\x}_{\left(\cdot \wedge \left(\underline{t}-\frac{1}{N}\right)\right)N}/\underline{t}N/Z_{1}		\right)	(	(\cdot \wedge \underline{t})N	)}{N^{\frac{1}{2}}}	\right)	\right]\right]\notag,
\end{align}
where 
  $Z_1(\cdot)$ is a random function independent of $Y_{\underline{\cdot}N}^{\x}$ such that $Z_1(s)=0$ for $0\leq s<1$, $P(Z_1(s)=1 \text{ for }s\geq 1)=P(Z_1(s)=-1 \text{ for }\geq 1)=\frac{1}{2}$. Iterating this,
\begin{align}
&E\left[	H_{\underline{t}}^{(N)}(\psi)			\right]\notag\\	
&=E\left[E_{Z_1,Z_2}	\left[\frac{1}{N}\sum_{\x\sim \underline{t}-\frac{2}{N} }\psi\left(	\frac{		\left(\left(Y^{\x}_{(\cdot \wedge \underline{t}-\frac{2}{N})N}\Big/\underline{t}N-2\Big/	Z_2\right)\Big/\underline{t}N-1\Big/Z_1\right)	\left(\left(\cdot\wedge \underline{t}\right)N\right)	}{N^{\frac{1}{2}}}	\right)\right]\right]		\notag\\
&=E_Y\left[\psi\left(\frac{Y_{\left(\left(\cdot\wedge \underline{t}\right)N\right)}}{N^{\frac{1}{2}}}\right)\right],\notag
\end{align}
where $Z_2$ is independent copy of $Z_1$ and $Y(\cdot)$ is the trajectory of simple random walk. 
Also, (\ref{bdd}) follows from the fact that $X_{\underline{t}}^{(N)}(1)$ is an ${\cal F}_{\underline{t}N}^{(N)}$-martingale and from the $L^{1}$ inequality for non-negative submartingales and from (\ref{srw}).

\end{proof}



\begin{proof}[Proof of Lemma \ref{lem3}]
We know $X^{(N)}_0(\phi)=\phi(0)$. Also, we have that for any $K>0$\begin{align}
	\left| C_{t}^{(N)}(\phi)-C_{s}^{(N)}(\phi)\right|		&\leq \int_{\underline{s}}^{\underline{t}}\left|X^{(N)}_{\underline{u}}\left(	A^N\phi		\right)\right|du\notag\\
	&\leq \sup_{\underline{u}\leq K}C(\phi)X^{(N)}_{\underline{u}}(1)|\underline{t}-\underline{s}|.\label{relative}
\end{align}
We can use the Arzela-Ascoli Theorem by (\ref{bdd}) and (\ref{relative}) so that $\left\{C^{(N)}_{\cdot}(\phi):N\in \mathbb{N}\right\}$ are $C$-relatively compact sequences in $D\left([0,\infty),\mathbb{R}\right)$.
\end{proof}


\begin{proof}[Proof of Lemma \ref{lem1}]
Let $\dis h_{N}(y)=E^{y}\left[	\left(	\phi\left(\frac{Y_{1}}{N^{\frac{1}{2}}}\right)-\phi\left(\frac{Y_{0}}{N^{\frac{1}{2}}}\right)		\right)^{2}			\right]$. First, we remark that \begin{align*}
\phi\left(\frac{Y_{\underline{s}N+1}^{\x}}{N^{\frac{1}{2}}}\right)-			\phi\left(\frac{Y_{\underline{s}N}^{\x}}{N^{\frac{1}{2}}}\right)-	\frac{1}{N}A^{N}\phi\left(\frac{Y_{\underline{s}N}^{\x}}{N^{\frac{1}{2}}}\right)		
\end{align*} are orthogonal for $\x\not=\x'\sim \underline{s}$. Since $M_t^{(s,N)}(\phi)$ is a martingale, we have that
 \begin{align}
&E\left[\left(		M_{K}^{(s,N)}	(\phi)	\right)^{2}\right]\notag\\
&\hspace{1em}=\frac{1}{N^2}\sum_{\underline{s}<K}E\left[	\left(\Delta M_{\underline{s}}^{(s,N)}(\phi)\right)^2			\right]\notag\\
&\hspace{1em}=\frac{1}{N^{2}}\sum_{\underline{s}<K}E\left[		\sum_{\x\sim \underline{s}}E\left[\left.		\left(		\phi\left(\frac{Y_{\underline{s}N+1}^{\x}}{N^{\frac{1}{2}}}\right)-			\phi\left(\frac{Y_{\underline{s}N}^{\x}}{N^{\frac{1}{2}}}\right)-	\frac{1}{N}A^{N}\phi\left(\frac{Y_{\underline{s}N}^{\x}}{N^{\frac{1}{2}}}\right)		
		\right)	^{2}	\right|	{\cal F}_{\underline{s}N}^{(N)}	\right]	\right]\notag\\
&\hspace{1em}\leq\frac{2}{N}\sum_{\underline{s}<K}E\left[	\frac{1}{N}	\sum_{\x\sim \underline{s}}\left(	h_{N}\left(\frac{Y_{\underline{s}N}^{\x}}{N^{\frac{1}{2}}}\right)+\frac{1}{N^{2}}\|	A^{N}\phi		\|^{2}		\right)		\right]\notag\\
&\hspace{1em}\leq 2E\left[	\int_{0}^{K}\left(X_{s}^{(N)}	({h_{N}})+\|A^{N}\phi\|_{\infty}^{2}N^{-2}X_{s}^{(N)}(1)\right)ds			\right]\notag\\
&\hspace{1em}\leq 2\left(E_Y\left[\int_{0}^{K}\left(\phi\left(\frac{Y_{\underline{s}N+1}}{N^{\frac{1}{2}}}\right)-\phi\left(\frac{Y_{\underline{s}N}}{N^{\frac{1}{2}}}\right)\right)	^{2}ds			\right]		+\frac{K}{N^{2}}\sup_{N}\|A^N\phi\|_{\infty}^{2}X_{0}^{(N)}	(1)\right)\notag\\
&\hspace{1em}\to 0,\notag
\end{align}
where we have used Lemma \ref{lem4} and the fact that $\sup_{N}\|A^N\phi\|_{\infty}<\infty$ for $\phi\in C_b^2(\mathbb{R})$ and $\{X^{(N)}_{\underline{t}}(1):0\leq \underline{t}\leq \underline{K}\}$ is martingale with respect to ${\cal F}^{(N)}_{\underline{t}N}$ in the last line.
\end{proof}



Next, we will check the conditions in Lemma \ref{lem2} for $M_{\cdot}^{(b,N)}(\phi)$ and $M_{\cdot}^{(e,N)}(\phi)$, that is, \begin{enumerate}
\item $\dis \left\{		\left\langle		M^{(b,N)}	(\phi)				\right\rangle_{\cdot}+\left\langle	M^{(e,N)}	(\phi)	\right\rangle_{\cdot}:N\in \mathbb{N}		\right\}$ is $C$-relatively compact in $D([0,\infty), \mathbb{R})$.
\item $\dis  \sup_{0\leq \underline{t}\leq K}\left|	M_{\underline{t}+\frac{1}{N}}^{(b,N)}(\phi)-M_{\underline{t}}^{(b,N)}	(\phi)		+M_{\underline{t}+\frac{1}{N}}^{(e,N)}(\phi)-M_{\underline{t}}^{(e,N)}	(\phi)			\right|\stackrel{P}{\rightarrow}0$ as $N\to \infty$ for all $K>0$.
\item $\dis \left\{		\left(M_{\underline{t}}^{(b,N)}(\phi)\right)^{2}+\left(M_{\underline{t}}^{(e,N)}(\phi)\right)^{2}+\left\langle		M^{(b,N)}		(\phi)\right\rangle_{\underline{t}}+\left\langle	M^{(e,N)}	(\phi)	\right\rangle_{\underline{t}}		:N\in \mathbb{N}		\right\}$ is uniformly integrable for all $\underline{t}$.
\end{enumerate}

As we verified that  $M^{(b,N)}(\phi)$ and $M^{(e,N)}(\phi)$ are orthogonal, 
 we have that \begin{align}
\left\langle			M^{(b,N)}(\phi)+M^{(e,N)}(\phi)	\right\rangle_{\cdot}=\left\langle			M^{(b,N)}(\phi)\right\rangle_{\cdot}+\left\langle M^{(e,N)}(\phi)	\right\rangle_{\cdot}.\notag
\end{align}

Moreover, since under fixed environment $\dis \{\xi(n,x):(n,x)\in \mathbb{N}\times \mathbb{Z}\}$, $V^{\x}$ and $V^{\y}$ are independent for $\x\not=\y$, we have that 
\begin{align}
&\left\langle	M^{(b,N)}	(\phi)\right\rangle_{\underline{t}}\notag\\
&=\sum_{\underline{s}<t}E\left[	\left.	\left(		M_{\underline{s}+\frac{1}{N}}^{(b,N)}(\phi)-M_{\underline{s}}^{(b,N)}(\phi)		\right)^2	\right|		{\cal F}_{\underline{s}N}^{(N)}	\right]\notag\\
&=\frac{1}{N^{2}}\sum_{\underline{s}<t}\sum_{\x\sim \underline{t}}E\left[\left.\phi\left(	\frac{Y^{\x}_{\underline{t}N+1}}{N^{\frac{1}{2}}}		\right)^{2}\right|		{\cal F}^{(N)}_{\underline{t}N}		\right]E\left[	\left.	\left(			V^{\x}-1-\frac{\beta\xi\left(\underline{t}N,Y^{\x}_{\underline{t}N}\right)}{N^{\frac{1}{4}}}	\right)^{2}\right|{\cal F}^{(N)}_{\underline{t}N}			\right]					\notag\\
&=\frac{1}{N}\sum_{\underline{s}<t}X^{(N)}_{\underline{t}}\left(\phi^{2}\right)\left(1-\frac{\beta^2}{N^{\frac{1}{2}}}\right)\left(1+{\cal O}(N^{-\frac{1}{2}})\right)\notag\\
&=\left(	1+{\cal O}(N^{-\frac{1}{2}})	\right)\int_{0}^tX_s^{(N)}(\phi^2)ds,\notag
\intertext{and}
&\left\langle	M^{(e,N)}(\phi)		\right\rangle_{\underline{t}}\notag\\
&=\sum_{\underline{s}<t}E\left[	\left.	\left(		M_{\underline{s}+\frac{1}{N}}^{(e,N)}(\phi)-M_{\underline{s}}^{(e,N)}(\phi)		\right)^2	\right|		{\cal F}_{\underline{s}N}^{(N)}	\right]\notag\\
&=\frac{\beta^2}{N^{2}}\sum_{\underline{s}<t}\sum_{\x,\tilde{\x}\sim \underline{t}}E\left[\left.\phi\left(\frac{Y^{\x}_{\underline{t}N+1}}{N^{\frac{1}{2}}}\right)\phi\left(\frac{Y^{\tilde{\x}}_{\underline{t}N+1}}{N^{\frac{1}{2}}}\right)\right|		{\cal F}^{(N)}_{\underline{t}N}		\right]\frac{{\bf 1}\left\{Y^{\x}_{\underline{t}N}=Y^{\tilde{\x}}_{\underline{t}N}\right\}}{N^{\frac{1}{2}}}\notag\\
&=\frac{\beta^2}{N^{2}}\sum_{\underline{s}<t}\sum_{\x,\tilde{\x}\sim \underline{t}}\phi\left(\frac{Y^{\x}_{\underline{t}N}}{N^{\frac{1}{2}}}\right)^{2}\frac{{\bf 1}\left\{	Y^{\x}_{\underline{t}N}=Y^{\tilde{\x}}_{\underline{t}N}		\right\}}{N^{\frac{1}{2}}}\left(		1+{\cal O}(N^{-\frac{1}{2}})	\right)\notag\\
&=\frac{1+{\cal O}(N^{-\frac{1}{2}})}{N}\beta^2\sum_{\underline{s}<t}\sum_{x\in\mathbb{Z}}\phi\left(		\frac{x}{N^{\frac{1}{2}}}		\right)^2\frac{\left(B_{\underline{s}N,x}^{(N)}\right)^2}{N^{\frac{3}{2}}}\notag\\
&=\left(1+{\cal O}\left(N^{-\frac{1}{2}}\right)\right)\beta^2\int_{0}^t\sum_{x\in\mathbb{Z}}\phi\left(		\frac{x}{N^{\frac{1}{2}}}		\right)^2\frac{\left(B_{\underline{s}N,x}^{(N)}\right)^2}{N^{\frac{3}{2}}}ds,\notag
\end{align}
where $|{\cal O}(N^{-\frac{1}{2}})|\leq C_{\phi}N^{-\frac{1}{2}}$ for a constant $C_{\phi}$ that depends only on $\phi$.


Therefore, we have that \begin{align}
&\left\langle		M^{(b,N)}	(\phi)		\right\rangle_{\underline{t}}+\left\langle		M^{(e,N)}	(\phi)	\right\rangle_{\underline{t}}-\left\langle		M^{(b,N)}		(\phi)	\right\rangle_{\underline{s}}-\left\langle		M^{(e,N)}	(\phi)	\right\rangle_{\underline{s}}\notag\\
&\leq C_{\phi}\left(\left\langle M^{(b,N)}(1)\right\rangle_{\underline{t}}+\left\langle M^{(e,N)}(1)\right\rangle_{\underline{t}}-\left\langle M^{(b,N)}(1)\right\rangle_{\underline{s}}-\left\langle M^{(e,N)}(1)\right\rangle_{\underline{s}}\right)\notag\\
&=C\left(\left\langle	X^{(N)}(1)			\right\rangle_{\underline{t}}-\left\langle	X^{(N)}(1)			\right\rangle_{\underline{s}}\right),\label{Ccpt}
\end{align}
where we remark that $\left\{X_{\underline{t}}^{(N)}(1):0\leq \underline{t}\right\}$ is a martingale with respect to ${\cal F}_{\underline{t}N}^{(N)}$.

We will prove $C$-relative compactness of (\ref{Ccpt}) by showing the following lemma.


\begin{lem}\label{lem12}For any $K>0$\begin{align}
&\sup_{N}E\left[		\left(X^{(N)}_{\underline{K}}(1)\right)^{2}	\right]<\infty.\notag,\\
\intertext{and for any $\ve>0$,}
&\lim_{\delta\to 0}\sup_{N\geq 1}P\left(	\sup_{0\leq s\leq K}\left(\left\langle X^{(N)}(1)\right\rangle_{\underline{s}+\delta}-\left\langle X^{(N)}(1)\right\rangle_{\underline{s}}			\right)>\varepsilon\right)=0\notag.
\end{align}

\end{lem}

\begin{proof}
We remark that for each $N$, $B_{n}^{(N)}$ is a martingale with respect to the filtration ${{\cal F}}^{(N)}_{n}$. 

Let $B_{n}^{(i,N)}$ be the total number of particles at time $n$ which are the descendants from $i$-th initial particle. Then, we remark that for $i\not= j$ \begin{align}
E\left[	B^{(i,N)}_{\lfloor	KN\rfloor}B^{(j,N)}_{\lfloor KN\rfloor}		\right]&=E\left[		E\left[	\left.		B_{\lfloor KN\rfloor}^{(i,N)}		\right|	{{\cal H}}	\right]		E\left[	\left.		B_{\lfloor KN\rfloor}^{(j,N)}		\right|	{{\cal H}}		\right]		\right]\notag\\
&=E_{Y^1Y^2}\left[	\left(1+\frac{\beta^2}{N^{\frac{1}{2}}}\right)^{\sharp \left\{i\leq \lfloor KN\rfloor :Y^1_{i}=Y^2_{i}			\right\}}		\right],\notag
\end{align}
where ${\cal H}$ is the $\sigma$-algebra generated by $\{\xi(n,x):(n,x)\in \mathbb{N}\times \mathbb{Z}\}$, and  $Y^1$ and $Y^2$ are independent simple random walks on $\mathbb{Z}$ starting from the origin.

On the other hand, \begin{align}
E\left[	\left(B^{(i,N)}_{\lfloor	KN\rfloor}\right)^{2}		\right]&=1+\sum_{k=1}^{\lfloor KN\rfloor-1 }cE_{Y^1Y^2}\left[	\left(1+\frac{\beta^2}{N^{\frac{1}{2}}}\right)^{\sharp \left\{k< i\leq \lfloor	KN	\rfloor :Y^1_{i}=Y^2_{i}\right\}}	:Y^1_{k}=Y^2_{k}		\right]+c\notag\\
&\leq \lfloor KN\rfloor E_{Y^1Y^2}\left[	\left(1+\frac{\beta^2}{N^{\frac{1}{2}}}\right)^{\sharp \left\{i\leq \lfloor KN\rfloor :Y^1_{i}=Y^2_{i}			\right\}}		\right],\notag
\end{align}
where $c=1-\frac{1}{N^{\frac{1}{2}}}<1$ \cite[Lemma 2.3]{NY}. Thus,  we have that \begin{align}
E\left[\left(X^{(N)}_{\underline{K}}(1)\right)^{2}\right]&\leq \frac{1}{N^{2}}\left(N(N-1)+N\lfloor KN\rfloor \right)E_{Y^1Y^2}\left[	\left(1+\frac{\beta^2}{N^{\frac{1}{2}}}\right)^{\sharp \left\{i\leq \lfloor KN\rfloor :Y^1_{i}=Y^2_{i}			\right\}}		\right]\notag\\
&\leq C(K) E_{Y^1Y^2}\left[	\left(1+\frac{\beta^2}{N^{\frac{1}{2}}}\right)^{\sharp \left\{i\leq \lfloor KN\rfloor :Y^1_{i}=Y^2_{i}			\right\}}		\right].\notag
\end{align}
Since $\dis E_{Y^1Y^2}\left[	\left(1+\frac{\beta^2}{N^{\frac{1}{2}}}\right)^{\sharp \left\{i\leq \lfloor KN\rfloor :Y^1_{i}=Y^2_{i}			\right\}}		\right]$ is bounded (Lemma \ref{lem13}),  we complete the proof.

Now, we turn to the proof of the latter part of the statement. Let $\delta>0$. It follows from the above argument that \begin{align*}
&\left\langle	X^{(N)}(1)			\right\rangle_{t}-\left\langle	X^{(N)}(1)			\right\rangle_{s}\\
&=\int_{s}^{t}\left(X_u^{(N)}(1)+\beta^2\sum_{x\in \mathbb{Z}}\frac{\left(B_{\lfloor uN\rfloor,x}^{(N)}\right)^2}{N^{\frac{3}{2}}}\right)du.
\end{align*}
We know that $\left|\int_s^tX_u^{(N)}(1)du\right|\leq \left(\sup_{u\leq K}X_u^{(N)}(1)\right)|t-s|$ and Lemma \ref{lem4} implies that this term converges in probability to $0$ as $|t-s|\to 0$ uniformly in $0\leq s\leq t\leq K$. So, it is enough to show that for any $\ve>0$\begin{align*}
\lim_{\delta\to 0}\sup_{N\geq 1}P\left(\sup_{0\leq s\leq K}\int_{s}^{s+\delta}\sum_{x\in \mathbb{Z}}	\frac{\left(B_{\lfloor uN\rfloor ,x}^{(N)}\right)^2}{N^{\frac{3}{2}}}du>\ve	\right)=0.
\end{align*}

We consider the segments $I_k^{\delta}=[2k\delta,2(k+1)\delta]$ for $0\leq k\leq \left\lfloor\frac{K}{2\delta}\right\rfloor$. Then, we have by Corollary \ref{cor1} that \begin{align}
&E\left[	\left(\int_{I_{\delta}^k}		\sum_{x\in \mathbb{Z}}	\frac{\left(B_{\lfloor uN\rfloor ,x}^{(N)}\right)^2}{N^{\frac{3}{2}}}du\right)^2	\right]\notag\\
&= \frac{1}{N^5}E\left[	\sum_{{s}=2k\delta N}^{2(k+1)\delta N}\sum_{t=2k\delta N}^{2(k+1\delta N)}	\sum_{x,y\in \mathbb{Z}}\left(	B_{\lfloor sN\rfloor ,x}^{(N)}		\right)^2\left(	B_{\lfloor tN\rfloor ,y}^{(N)}		\right)^2				\right]\notag\\
&\leq \frac{1}{N^5}\left(	\sum_{s=2k\delta N}^{2(k+1)\delta N}\sum_{x\in \mathbb{Z}}E\left[				\left(	B_{\lfloor sN\rfloor ,x}^{(N)}		\right)^4			\right]^{\frac{1}{2}}			\right)^2.\label{l4}
\end{align}
Corollary \ref{cor1} implies that \begin{align}
E\left[				\left(	B_{\lfloor sN\rfloor ,x}^{(N)}		\right)^4			\right]\leq (s\vee 1)^4N^4E_{Y^1Y^2Y^3Y^4}\left[	\left(1+\frac{7\beta^2}{N^{\frac{1}{2}}}\right)^{\sharp\{1\leq i\leq sN:Y^{a}_i=Y_i^{b},a,b\in \{1,2,3,4\}\}}		:\begin{subarray}{l}Y^a_{\lfloor sN\rfloor}=x,\\	a\in\{1,2,3,4\}\end{subarray}\right],\label{m4}
\end{align}
where we have used that for $N$ large enough, $\dis E\left[\left(1+\frac{\beta\xi(0,0)}{N^{\frac{1}{4}}}\right)^4\right]\leq 1+\frac{7\beta^2}{N^{\frac{1}{2}}}$. H\"{o}lder's inequality and Lemma \ref{lem13} imply that \begin{align*}
(\ref{m4})&\leq (s\vee 1)^4N^4E_{Y^1Y^2}\left[	\left(	1+\frac{7\beta^2}{N^{\frac{1}{2}}}		\right)^{6\sharp\{1\leq i\leq sN:Y^1_i=Y^2_i\}}:Y^1_{\lfloor sN\rfloor }=Y^{2}_{\lfloor sN\rfloor }=x			\right]P_{Y^1}\left(	Y^1_{\lfloor sN\rfloor }=x	\right)^2\\
&\leq \frac{(s\vee 1)^4N^4}{\left(sN\vee 1\right)^{\frac{1}{2}}}P_{Y^1}\left(	Y^1_{\lfloor sN\rfloor }=x	\right)^3.
\end{align*}
Thus, local limit theorem implies that \begin{align*}
&\text{(\ref{l4})}\leq \frac{C}{N}\left(\sum_{s=2k\delta N}^{2(k+1)\delta N}\sum_{x\in \mathbb{Z}}\frac{(K\vee 1)^2}{\left(sN\vee 1\right)^{\frac{1}{4}}}\frac{1}{(sN\vee 1)^\frac{1}{4}}P_{Y^1}\left(	Y^1_{\lfloor sN\rfloor }=x	\right)\right)^2\\
&\leq \frac{CK^4}{N}\left(		\sqrt{2(k+1)\delta N}-\sqrt{2k\delta N}			\right)^2.
\end{align*}
Thus, we obtained that \begin{align*}
P\left(		\int_{I_k^{\delta}}\sum_{x\in \mathbb{Z}}	\frac{\left(B_{\lfloor uN\rfloor ,x}^{(N)}\right)^2}{N^{\frac{3}{2}}}du>\ve	\right)\leq \frac{CK^4\delta }{\ve^2(\sqrt{2(k+1)}+\sqrt{2k})^2}.
\end{align*}
Since for each $0\leq s\leq K$, there is some $k$ such that $[s,s+\delta]\subset I_{k}^{\delta}\cup I_{k+1}^{\delta}$, we have that \begin{align*}
\sup_{N\geq 1}P\left(	\sup_{0\leq s\leq K}\int_{s}^{s+\delta}\sum_{x\in \mathbb{Z}}	\frac{\left(B_{\lfloor uN\rfloor ,x}^{(N)}\right)^2}{N^{\frac{3}{2}}}du>\ve					\right)&\leq 2\sum_{k=0}^{\frac{K}{\delta}}\frac{CK^4\delta }{\ve^2(\sqrt{2(k+1)}+\sqrt{2k})^2}\\
&\leq 2\frac{CK^4\delta \log\frac{K}{\delta}}{\ve^2}\to 0\, \, \text{as }\delta\to 0.
\end{align*}
\end{proof}

Also, we prove the following lemmas to check the conditions (1)-(3).

\begin{lem}\label{lem5}For $\phi\in C_{b}^{2}(\mathbb{R})$,\begin{align}
\lim_{N\to\infty }E\left[	\sum_{\underline{t}\leq K}|\Delta M^{(b,N)}_{\underline{t}}(\phi)+\Delta M^{(e,N)}_{\underline{t}}(\phi)|^{4}		\right]=0 \text{\hspace{3em} for all $K>0$}.\notag
\end{align}
\end{lem}

\begin{lem}\label{lem6} For $\phi\in C_b^2(\mathbb{R})$,
\begin{align}
&\sup_{N}E\left[	\sup_{\underline{t}\leq K}\left|	M_{\underline{t}}^{(b,N)}(\phi)+M_{\underline{t}}^{(e,N)}(\phi)		\right|^{4}			\right]<\infty&\text{\hspace{3em}for all $K>0$},\notag
\intertext{and }
&E\left[	\left(\left\langle		M^{(b,N)}(\phi)+M^{(e,N)}(\phi)	\right\rangle_{K}\right)^2		\right]<\infty\ \ &\text{for all $K>0$.}\notag
\end{align}

\end{lem}

If we prove these lemmas, then we can verify the condition of Theorem \ref{thm2} (ii).

\begin{proof}[Proof of the $C$-relatively compactness of $\{X_{\cdot}^{(N)}(\phi):N\in \mathbb{N}\}$]
When we look at the process $\left\{X_{\cdot}^{(N)}(\phi)\right\}$, it is divided into some processes, $X_0^{(N)}(\phi)$, $M^{(b,N)}_{\cdot}(\phi)$, $M^{(e,N)}_{\cdot}(\phi)$, $M^{(b,N)}_{\cdot}(\phi)$, and $C_{\cdot}^{(N)}(\phi)$. 

We know that $M^{(b,N)}_{\cdot}(\phi)$ and $X_0^{(N)}(\phi)$ converges to constant by Assumptions and Lemma \ref{lem1}. $C$-relative compactness of $C^{(N)}_{\cdot}(\phi)$ has been proved in Lemma \ref{lem3}.

Arzela-Ascoli's theorem and Lemma \ref{lem12} imply that $\dis \left\{\left\langle 		M^{(b,N)}(\phi)+M^{(e,N)}(\phi)	\right\rangle_{\cdot}:N\in\mathbb{N}\right\}$ is $C$-relatively compact in $D\left([0,\infty),\mathbb{R}\right)$.
Also, (\ref{Pconv}) follows from Lemma \ref{lem5} .
The uniform integrability of $\dis \left\{	\left(M_{\underline{t}}^{(b,N)}(\phi)+M_{\underline{t}}^{(e,N)}(\phi)\right)^{2}	+\left\langle		M^{(b,N)}(\phi)+M^{(e,N)}(\phi)	\right\rangle_{\underline{t}}		\right\}$ has been shown by Lemma \ref{lem12} and Lemma \ref{lem6}. Thus, we have checked all conditions in Lemma \ref{lem2} so that $\dis \left\{M^{(b,N)}_{\cdot}(\phi)+M^{(e,N)}_{\cdot}(\phi),\left\langle M^{(b,N)}(\phi)+M^{(e,N)}(\phi)\right\rangle_{\cdot}\right\}$ is $C$-relatively compact in $D([0,\infty),\mathbb{R})$. 

Thus, $\left\{		X_{\cdot}^{(N)}(\phi)	\right\}$ is $C$-relatively compact in $D([0,\infty),\mathbb{R})$ for each $\phi\in C_b^2(\mathbb{R})$.

\end{proof}

To prove Lemma \ref{lem5}, we will use the following proposition (see \cite{Bur}).


\begin{prop}\label{prop1}
Let $\phi:\mathbb{R}_{\geq 0}\to\mathbb{R}_{\geq 0}$ is continuous, increasing, $\phi(0)=0$ and $\phi(2\lambda)\leq c_{0}\phi(\lambda)$ for all $\lambda\geq 0$. $(M_{n}, {\cal F}_{n}) $ is a martingale, $M^{*}_{n}=\sup_{k\leq n}|M_{k}|$, $\dis \left\langle	M		\right\rangle_{n}=\sum_{i=1}^{n}E\left[\left.	\left(M_{k}-M_{k-1}\right)^{2}	\right|{\cal F}_{k-1}\right]+E[M_{0}^{2}]$, and $d^{*}_{n}=\max_{1\leq k\leq n}|M_{k}-M_{k-1}|$. Then, there exists $c=c(c_{0})$ such that \begin{align}
E\left[	\phi\left(	M_{n}^{*}	\right)		\right]\leq cE\left[\phi\left(			\left\langle M\right\rangle^{1/2}_{n}	\right)+\phi\left(d^{*}_{n}\right)\right].\notag
\end{align}
\end{prop}


\begin{proof}[Proof of Lemma \ref{lem5}]
It is enough to show that \begin{align}
\lim_{N\to\infty} E\left[	\sum_{\underline{t}\leq K}\left|		\Delta M^{(b,N)}_{\underline{t}}(\phi)	\right|^{4}+	\left|		\Delta M^{(e,N)}_{\underline{t}}(\phi)	\right|^{4}	\right]=0 \text{\hspace{3em} for all $K>0$.}\notag
\end{align}
Conditional on ${\cal G}_{\underline{t}N}^{(N)}$, $\dis \Delta M^{(b,N)}_{\underline{t}}(\phi)$ is a sum of mean $0$ independent random variables; $\dis W^{(b,\x,N)}:=\frac{1}{N}\phi\left(\frac{Y_{\underline{t}N+1}^{\x}}{N^{\frac{1}{2}}}\right)\left(	V^{\x}-1-\frac{\beta\xi\left(	\underline{t}N,Y^{\x}_{\underline{t}N}		\right)}{N^{\frac{1}{4}}}			\right)$. Applying Proposition \ref{prop1} into $\dis \sum_{\x\sim \underline{t}} W^{(b,\x,N)}$, we have
 \begin{align}
E\left[	\left.\left(	\sup_{i\leq B_{\underline{t}N}^{(N)}}\sum_{k=1}^{i}W^{(b,\x_{k},N)}	\right)^{4}			\right|{\cal G}_{\underline{t}N}^{(N)}\right]&
\leq c\left(		\sum_{i\leq B_{\underline{t}N}^{(N)}}	\left(\frac{C_{1}(\phi)\left(1-{\cal O}(N^{-1/2})\right)}{N^{2}}\right)^{2}	
+\left(\frac{C_{2}(\phi)}{N}	\right)^{4}	
\right).\notag
\end{align} 
Thus, \begin{align}
E\left[		\sum_{\underline{t}\leq K} \left|		\Delta M^{(b,N)}_{\underline{t}}(\phi)	\right|^{4}	\right]&\leq c\left(		\frac{C_{1}(\phi)^{2}(1-{\cal O}(N^{-1/2}))}{N^{4}}\cdot(KN)\cdot E[NX_{\underline{t}}^{(N)}(1)]	+KN\cdot \frac{C_{2}(\phi)^{4}}{N^{4}}	\right)\notag\\
&\to 0.\notag
\end{align}

Next, we will  prove that \begin{align}
\lim_{N\to\infty}E\left[	\sum_{\underline{t}\leq K}\left|	\Delta M_{\underline{t}}^{(e,N)}(\phi)		\right|^{4}			\right]=0\ \ \text{for all }K>0.\notag
\end{align}
It is clear that for $\phi\in C_{b}^{2}(\mathbb{R})$\begin{align}
E\left[			\left|			\Delta M_{\underline{t}}^{(e,N)}(\phi)		\right|^{4}	\right]&\leq C(\phi)E\left[		\sum_{x,y\in\mathbb{Z}}2\frac{\left(B^{(N)}_{\underline{t}N,x}\right)^{2}\left(B^{(N)}_{\underline{t}N,y}\right)^{2}}{N^{5}}		\right].\notag
\end{align}

Then, it follows from Corollary \ref{cor1} and the similar argument in the proof of Lemma \ref{lem12} that\begin{align*}
&\frac{E\left[	\left(B_{\underline{t}N,x}^{(N)}\right)^2	\left(B_{\underline{t}N,y}^{(N)}\right)^2			\right]}{N^5}\\
&\leq \frac{C(\underline{t}\vee 1)^4}{N}E_{Y^1Y^2Y^3Y^4}\left[	\left(	1+\frac{7\beta^2}{N^\frac{1}{2}{}}		\right)^{\sharp\{	1\leq i\leq \underline{t}N:Y^a_i=Y^b_i,a,b\in\{1,2,3,4\}\}}		:\begin{subarray}{l}Y^1_{\underline{t}N}=Y_{\underline{t}N}^2=x\\	Y^3_{\underline{t}N}=Y^b_{\underline{t}N}=y	\end{subarray}	\right]\\
&\leq \frac{C(\underline{t}\vee 1)^4}{N}\prod_{\begin{subarray}{c}a,b\in \{1,2,3,4\}	\\ a\not= b	\end{subarray}	}E_{Y^1Y^2Y^3Y^4}\left[	\left(	1+\frac{7\beta^2}{N^\frac{1}{2}{}}		\right)^{6\sharp\{	1\leq i\leq \underline{t}N:Y^a_i=Y^b_i\}}		:\begin{subarray}{l}Y^1_{\underline{t}N}=Y_{\underline{t}N}^2=x\\	Y^3_{\underline{t}N}=Y^b_{\underline{t}N}=y	\end{subarray}	\right]^{\frac{1}{6}}\\
&\leq \frac{C(\underline{t}\vee 1)^4}{N\sqrt{\underline{t}N}}P_{Y^1}\left(Y^1_{\underline{t}N}=x		\right)P_{Y^1}\left(Y^1_{\underline{t}N}=y	\right)\left(P_{Y^1}\left(Y^1_{\underline{t}N}=x		\right)\wedge P_{Y^1}\left(Y^1_{\underline{t}N}=y	\right)\right).
\end{align*}
Thus, we have that \begin{align*}
E\left[			\left|			\Delta M_{\underline{t}}^{(e,N)}(\phi)		\right|^{4}	\right]\leq C(\phi)(K\vee 1)^4\sum_{\underline{t}\leq K}\frac{1}{N\cdot \underline{t}N}\to 0,
\end{align*}
as $N\to \infty$.

\end{proof}


\begin{proof}[Proof of Lemma \ref{lem6}]
We apply Proposition \ref{prop1} into martingale $\dis M^{(b,N)}_{\underline{t}}(\phi)+ M^{(e,N)}_{\underline{t}}(\phi)$. Then, we have that \begin{align}
E\left[\sup_{\underline{t}\leq K}	\left(	M^{(b,N)}_{\underline{t}}(\phi)+	M^{(e,N)}_{\underline{t}}(\phi)\right)^{4}\right]\leq &c(\phi)\left(E\left[	\left(\left\langle	M^{(b,N)}(1)\right\rangle_{K}+\left\langle M^{(e,N)}(1)\right\rangle_{K}\right)^{2}	\right]\right.\notag\\
&\hspace{2em}\left.+	\sum_{\underline{t}\leq K}\left( \left|		\Delta M_{\underline{t}}^{(b,N)}(1)	\right|^{4}+\left|		\Delta M^{(e,N)}_{\underline{t}}(1)	\right|	^{4}\right)	\right)\notag.
\end{align}
The second term in the right hand side goes to $0$ as $N\to\infty$ by Lemma \ref{lem5}. The first term is bounded above by \begin{align}
CE\left[\sum_{\underline{s},\underline{t}\leq K}\left(\frac{X_{\underline{s}}^{(N)}(1)X_{\underline{t}}^{(N)}(1)}{N^{2}}+\beta^4\sum_{x,y\in\mathbb{Z}}\frac{\left(B_{\underline{ t}N,x}^{(N)}\right)^{2}}{N^{\frac{3}{2}}}\frac{\left(B_{ \underline{s}N,y}^{(N)}\right)^{2}}{N^{\frac{3}{2}}}\right)\right].\notag
\end{align} 
Since $X_{\underline{t}}^{(N)}(1)$ is a martingale, $\dis E\left[X_{\underline{s}}^{(N)}(1)X_{\underline{t}}^{(N)}(1)\right]=E\left[X_{\underline{s}}^{(N)}(1)X_{\underline{s}}^{(N)}(1)\right]$ for $\underline{s}\leq \underline{t}$. Thus, \begin{align}
E\left[\sum_{\underline{s},\underline{t}\leq K}\frac{X_{\underline{s}}^{(N)}(1)X_{\underline{t}}^{(N)}(1)}{N^{2}}\right]\leq K^{2}E\left[	\left(X_{\underline{K}}^{(N)}(1)\right)^{2}		\right]\notag
\end{align}
is bounded in $N$ for all $K$ by Lemma \ref{lem12}.

Also, we know that from the proof of Lemma \ref{lem12} that 
\begin{align}
\sum_{\underline{s},\underline{t}\leq K}E\left[		\sum_{x,y\in\mathbb{Z}}\frac{\left(B_{\underline{s}N,x}^{(N)}\right)^{2}\left(B_{\underline{t}N,y}^{(N)}\right)^{2}}{N^{5}}				\right]&\leq \frac{CK^4}{N}
\left(	\sqrt{KN}			\right)^2<\infty.\notag
\end{align}

\end{proof}


In the end of this subsection, we complete the proof of the tightness by  checking the condition (i) in Theorem \ref{thm2}. The proof follows the one in \cite[p155]{Per}
\begin{proof}[Check for  (i) in Theorem \ref{thm2}]
Let $\varepsilon,T>0$ and $\eta(\varepsilon)>0$ ($\eta$ will be chosen later). Let $K_{0}\subset D([0,\infty),\mathbb{R})$ be a compact set such that $\dis \sup_{N}P\left(  \frac{Y_{\cdot N}}{N^{\frac{1}{2}}}\in K_{0}^{c} \right)<\eta$. Let $K_T=\{y_{t},y_{t-}:t\leq T,y\in K_{0}\}$. Then, $K_T$ is compact in $\mathbb{R}$. Clearly, \begin{align}
\sup_{N}P\left(		\frac{Y_{Nt}}{N^{\frac{1}{2}}}\in K_T^{c}		\text{ for some }t\leq T\right)<\eta.\notag
\end{align}
Let \begin{align}R_{t}^{(N)}&=H_{t}^{(N)}\left(		y:y(s)\in K_T^{c}\text{ for some }s\leq t			\right)\notag\\
&=\frac{1}{N}\sum_{\x\sim \underline{t}}\sup_{\underline{s}\leq \underline{t}}{\bf 1}_{K_T^{c}}\left(\frac{Y^{\x}_{\underline{s}N}}{N^{\frac{1}{2}}}\right).\notag\end{align}
First, we will claim that $R_{t}^{(N)}$ is an ${\cal F}_{\underline{t}N}^{(N)}$-submartingale. Clearly, $R_{\cdot}^{(N)}$ is constant on $[\underline{t},\underline{t}+\frac{1}{N})$. So, it is enough to show that \begin{align}
E\left[	\left.R_{\underline{t}+\frac{1}{N}}^{(N)}-R_{\underline{t}}^{(N)}	\right| {\cal F}_{\underline{t}N}^{(N)}	\right]\geq 0\ \text{ a.s.}\label{submar}
\end{align}
We have \begin{align}
R_{\underline{t}+\frac{1}{N}}^{(N)}-R_{\underline{t}}^{(N)}&=\frac{1}{N}\sum_{\x\sim \underline{t}}\sup_{\underline{s}\leq \underline{t}+\frac{1}{N}}{\bf 1}_{K_T^{c}}\left(\frac{Y^{\x}_{\underline{s}N}}{N^{\frac{1}{2}}}\right)V^{\x}-\sup_{\underline{s}\leq \underline{t}}{\bf 1}_{K_T^{c}}\left(\frac{Y^{\x}_{\underline{s}N}}{N^{\frac{1}{2}}}\right)\notag\\
&\geq \frac{1}{N}\sum_{\x\sim \underline{t}}\left(V^{\x}-1\right)\sup_{\underline{s}\leq \underline{t}}{\bf 1}_{K_T^{c}}\left(\frac{Y^{\x}_{\underline{s}N}}{N^{\frac{1}{2}}}\right).\notag
\end{align}
The conditional expectation of the last term with respect to ${\cal F}_{\underline{t}N}^{(N)}$ is equal to $0$. Thus,  (\ref{submar}) is proved. Now we apply $L^{1}$-inequality for submartingale into $R^{(N)}_{\cdot}$ so that\begin{align}
P\left(\sup_{\underline{s}\leq T }X_{\underline{s}}^{(N)}\left(K_T^{c}\right)>\varepsilon \right)&\leq P\left(			\sup_{t\leq T}R^{(N)}_{t}>\varepsilon \right)\notag\\
&\leq \varepsilon ^{-1}E[R_{\underline{T}}^{(N)}]\notag\\
&\leq \varepsilon^{-1} P\left(\frac{Y_{sN}}{N^{\frac{1}{2}}}\in K_T^{c},\ \text{for some }s\leq T\right)\leq \varepsilon\notag
\end{align}
by taking $\eta(\varepsilon)=\varepsilon^{2}$.
\end{proof}

\subsection{Identification of the limit point process}
From the lemmas in section \ref{tight}, we know that for $\phi\in C^{2}_b(\mathbb{R})$, each term of \begin{align}
&Z^{(N)}_{\underline{t}}(\phi)=X^{(N)}_{\underline{t}}(\phi)-\phi(0)-\int_{0}^{\underline{t}}X^{(N)}_{\underline{s}}(A^N\phi)ds,\label{mar1}
\intertext{and}
&\left\langle Z^{(N)}(\phi)\right\rangle_{\underline{t}}=\left\langle M^{(b,N)}(\phi)\right\rangle_{\underline{t}}+\left\langle M^{(e,N)}(\phi)\right\rangle_{\underline{t}}+\left\langle		M^{(s,N)}(\phi)		\right\rangle_{\underline{t}}\notag
\end{align}
are $C$-relatively compact in $D([0,\infty),\mathbb{R})$ and we found by from Lemma \ref{lem2} that the limit points satisfy \begin{align}
Z_{t}(\phi)=X_{t}(\phi)-\phi(0)-\int_{0}^{t}\frac{1}{2}X_{s}(\Delta \phi)ds\notag
\intertext{and }
\left\langle		Z(\phi)		\right\rangle_{t}=\int_{0}^tX_s(\phi^2)ds+M^{(e)}_{t}(\phi),\notag
\end{align}
where $M^{(e)}_{t}(\phi)$ is a limit point of $M^{(e,N)}_{\underline{t}}(\phi)$. 
Therefore, we need to identify $M^{(e)}_{{t}}(\phi)$.

First, we give an approximation of  $X_{t}^{(N)}$ by some measure valued processes which have densities. 
For $(t,y)\in \mathbb{R}_{\geq 0}\times \mathbb{R}$, we define $u^{(N)}(t,y)$ by \begin{align}
u^{(N)}(t,y)=\frac{B^{(N)}_{tN,x}}{2\sqrt{N}}\ \text{ \hspace{3em}for $\underline{t}\leq t<\underline{t}+\frac{1}{N}$ and $\displaystyle y\in\left[\frac{x-1}{N^{\frac{1}{2}}},\frac{x+1}{N^{\frac{1}{2}}}\right),\ x\in\mathbb{Z} .$ }
\notag
\end{align}
Actually, integrating  $u^{(N)}(t,y)$ over $\left[\frac{x-1}{N^{\frac{1}{2}}},\frac{x+1}{N^{\frac{1}{2}}}\right)$ for each $x\in \mathbb{Z}$, they coincide with $\frac{B_{tN,x}^{(N)}}{N}$. Thus, we can regard $u^{(N)}(t,y)$ as an approximation of $X^{(N)}_{\cdot}$.

Also, $\left\langle M^{(e,N)}(\phi)\right\rangle_{\underline{t}}$ can be rewritten as \begin{align}
\left\langle M^{(e,N)}(\phi)\right\rangle_{\underline{t}}&=\int_{0}^{\underline{t}}\sum_{x\in\mathbb{Z}}\phi\left(\frac{x}{N^{\frac{1}{2}}}\right)^{2}\frac{\beta^2\left(B^{(N)}_{\lfloor sN\rfloor ,x }\right)^2}{N^{\frac{3}{2}}}ds\notag\\
&=2\beta^2{(1+{\cal O}(N^{-\frac{1}{2}}))}\int_{0}^{t}\int_{y\in\mathbb{R}}\phi(y)^{2}u^{(N)}(s,y)^2dyds.\notag
\end{align}
Therefore, we can conjecture that the limit point $M^{(e)}_{t}(\phi)$ is \begin{align}
2\beta^2\int_{0}^t\int_{y\in\mathbb{R}}\phi^2(y)u(s,y)^2dsdy\label{limitpoint}
\end{align}
if $u^{(N)}\Rightarrow u$ for some $u(s,y)$ in some sense. In the following, we will check that (\ref{limitpoint}) is true.

We denote by $\tilde{X}^{(N)}_{t}$ new measure-valued processes associated to $u^{(N)}(\cdot,\cdot)$, that is for $\phi\in C_{b}^{2}(\mathbb{R})$,\begin{align}
\tilde{X}_{t}^{(N)}(\phi)=\int_{\mathbb{R}}\phi(x)u^{(N)}(t,x)dx.\notag
\end{align}
Then, it is clear that for $C^{2}_b(\mathbb{R})$ and for any $K>0$\begin{align}
\limsup_{N\to\infty}E\left[\sup_{t<K}\left|\tilde{X}^{(N)}_{t}(\phi)-X_{t}^{(N)}(\phi)\right|\right]=0.\notag
\end{align}
Thus, $\left\{\tilde{X}^{(N)}_{\cdot}:N\in\mathbb{N}\right\}$ is $C$-relative compact in $D([0,\infty),{\cal M}_{F}(\mathbb{R}))$ and there are subsequences which weakly converges to $X_{\cdot}$, where $X_{\cdot}$  is the one given in (\ref{mar1}).

We will prove the following lemmas:
\begin{lem}\label{lem15}
Let $X_{\cdot}$ be a limit point of the sequence $\{X^{(N)}_{\cdot}:N\in\mathbb{N}\}$. Then, the measure valued process $\{X_{t}(\cdot):0\leq t<\infty\}$ is almost surely absolutely continuous for all $t>0$, that is there exists an adapted  Borel-measurable-function-valued process $\{u_t:t>0\}$ such that \begin{align}
X_{t}(dx)=u_{t}(x)dx,\ \ \text{for all }t>0,\ P\text{-a.s.}\notag
\end{align}
\end{lem}

Define a sequences of measure valued processes $\dis \left\{\mu_{\cdot}^{(N)}(dx):N\in \mathbb{N}\right\}$ by \begin{align}
\mu_{t}^{(N)}(dx)=2\beta^2\int_{0}^t\left(u^{(N)}(s,x)\right)^{2}dxds.\notag
\end{align}

\begin{lem}\label{lem16}
For any $\ve>0$ and for any $T>0$, there exists a compact set $K^{\ve,T}\subset \mathbb{R}$ such that \begin{align}
\sup_{N}P\left(\sup_{t\leq T}\mu_{t}^{(N)}\left(\left(K^{\ve,T}\right)^c\right)>\ve\right)<\ve.\notag
\end{align}
\end{lem}

By using Lemma \ref{lem15} and Lemma \ref{lem16}, we can can identify the limit point process as follows:

\begin{proof}[Identification of the limit point processes]
We will verify that  if $\dis X^{(N_{k})}_{\cdot}(dx)\Rightarrow u(\cdot,x)dx$ as $N_{k}\to \infty$, then \begin{align}
\mu^{(N_k)}_{t}(dx)\Rightarrow \left(2\beta^2\int_{0}^tu(s,x)^{2}ds\right)dx.\label{u2}
\end{align}

Actually, $\dis\left\{\left(\mu^{(N)}_{t}(\cdot)\right)_{t\in[0,\infty)}:N\in\mathbb{N}\right\}$ are $C$-relatively compact in $D\left([0,\infty),{\cal M}_{F}(\mathbb{R})\right)$ if the conditions in Theorem \ref{thm2} are satisfied. However, we have already checked them in the proof of the tightness of $\{X^{(N)}_{\cdot}:N\in\mathbb{N}\}$ and Lemma \ref{lem16}.			 Thus, for any $\phi\in C_{b}^2(\mathbb{R})$, \begin{align}
\mu_{t}^{(N_{k})}\left(\phi\right)\Rightarrow \mu_t\left(\phi\right)\ \ \text{for subsequences }N_k\to \infty.\notag
\end{align} 
Also, we may consider this convergence is almost surely by Skorohod representation theorem, that is \begin{align}
\lim_{k\to\infty}\mu_{t}^{(N_{k})}\left(\phi\right)=\mu_{t}(\phi), \ \ \ \text{a.s.}\label{conv} 
\end{align}


Let $G_{N}(B,m)$ be the distributions of $u^{(N)}(t,x)$ for $B\in {\cal B}(\mathbb{R}_{\geq 0}\times \mathbb{R})$ and $m\in[0,\infty)$, that is \begin{align}
G_{N}(B,m)=\left|	\left\{(t,x)\in B:u^{(N)}(t,x)\leq m		\right\}	\right|,\notag
\end{align}
where $|\cdot|$ represents Lebesgue measure on $\mathbb{R}_{\geq 0}\times \mathbb{R}$.
Especially, \begin{align*}
G_N([0,t]\times \mathbb{R},m)=\frac{2}{N^{\frac{3}{2}}}\sharp\left\{	(n,x):n\leq \{0,\cdots,\lfloor tN\rfloor \},x\in\mathbb{Z}, B_{n,x}\leq 2m\sqrt{N}		\right\}.
\end{align*}
Then, the convergence of $u^{(N)}_{t}(\cdot)$ in (\ref{conv}) is equivalent to the convergence of the distributions $G_{N}(\cdot,\cdot)$.

Let $\mu^{(M,N)}_{t}(\cdot)$ be the truncated measure of $\mu_{t}^{(N)}(\cdot)$ for $M>0$, that is \begin{align}
\mu_{t}^{(M,N)}(dx)=\left(2\beta^2\int_{0}^t\left(u^{(N)}(s,x)\wedge M\right)^{2}ds\right)dx.\notag
\end{align}
Then, it is clear that for any bounded function $C^{2}_{b,+}(\mathbb{R})$\begin{align}
&\int_{0}^{t}\int_{\mathbb{R}}\phi(x)\left(u^{(N)}(s,x)\wedge M\right)^{2}dxds\notag\\
&\hspace{3em}=2\int_{0}^{t}\int_{\mathbb{R}}\int_{0}^{M}\phi(x)m^{2}G_{N}(dsdxdm)\notag\\
&\hspace{3em}+2\int_0^t\int_{\mathbb{R}}\int_M^{\infty}{\bf 1}_{\{u^{(N)}{(s,x)}>M\}}\phi(x)M^2G_N(dsdxdm).\notag
\end{align}
The last term converges to $0$ in probability as $N\to \infty$ and then $M\to \infty$. Indeed, we have that \begin{align*}
0\leq &\int_0^t\int_{\mathbb{R}}\int_{M}^{\infty}{\bf 1}\{u^{(N)}(s,x)>M\}\phi(x)M^2G_N(dsdxdm)\\
&\leq C(\phi)\frac{\left(B_{n,x}^{(N)}\right)^2}{N^{\frac{5}{2}}}\sharp\left\{	(n,x):n\leq \{0,\cdots,\lfloor tN\rfloor \},x\in\mathbb{Z}, B_{n,x}\geq 2M\sqrt{N}		\right\},
\end{align*}
and the last term converges to 0 in probability by Lemma \ref{lem6}.
Also, as $N_k\to \infty$  $\dis \int_{0}^{t}\int_{\mathbb{R}}\int_{0}^{M}\phi(x)m^{2}G_{N_k}(dsdxdm)$ converges almost surely to \begin{align}
\int_{0}^{t}\int_{\mathbb{R}}\int_{0}^{M}\phi(x)m^{2}G(dsdxdm)=\int_{0}^{t}\int_{\mathbb{R}}\phi(x)u(s,x)^2{\bf 1}\{u(t,x)\leq M\}dxds,\notag
\end{align}
where $G(\cdot,\cdot,\cdot)$ is the distribution of $u(t,x)$.
Thus, we have that for any $\phi\in C_{b,+}^{2}(\mathbb{R})$\begin{align}
\int_{0}^{t}\int_{\mathbb{R}}\phi(x)u(s,x)^{2}dxds 		&=\lim_{M\to\infty}\lim_{N_k\to\infty}\int_{0}^t 			\int_{\mathbb{R}}\int_0^M\phi(x)m^2G_{N_k}(dsdxdm)			\notag\\
&\leq \lim_{M\to\infty}\lim_{N_k\to\infty}\int_{0}^{t}\int_{\mathbb{R}}\phi(x) \left(u^{(N_k)}(t,x)\wedge M\right)^{2}dxds\notag\\
&\leq  \mu_{t}(\phi),\ \ \ \text{a.s.}\notag
\end{align}
Also, we know that  for bounded function $\phi\in C_{b,+}^{2}(\mathbb{R})$, for any $ t>0$ and for any $\ve>0$ \begin{align}
&\lim_{M\to\infty}\sup_{N}P\left(\left|\int_{0}^t\int_{\mathbb{R}}\phi(x)\left(\left(u^{(N)}(s,x)\right)^2-\left(u^{(N)}(s,x)\wedge M\right)^{2}\right)dxds\right|			>\ve		\right)\notag\\
&\leq 	\lim_{M\to\infty}		\sup_{N}P\left(	\left|		\int_0^t\int_{\mathbb{R}}\int_{M}^{\infty}\phi(x)m^2G_{N}(dsdxdm)		\right|>\ve		\right)	\notag\\
&=0,\notag
\end{align}
by Lemma \ref{lem6}.
Thus, for any bounded function $\phi\in C_{b,+}^2(\mathbb{R})$\begin{align}
\mu_t(\phi)&=\lim_{N_k\to \infty}2\beta^2\int_0^t\int_{\mathbb{R}}\phi(x)\left(u^{(N_k)}(t,x)\right)^{2}dxds\notag\\
&\leq 2\beta^2\int_0^t\int_{\mathbb{R}}\phi(x)u(t,x)^{2}dxds,\ \ \text{in probability.}\notag
\end{align}
This is true for $\phi\in C_{b}^2(\mathbb{R})$. Thus, we have proved (\ref{u2}).
\end{proof}


\begin{proof}[Proof of Lemma \ref{lem16}]
First, we remark that $M^{(e,N)}_{\underline{t}}(\phi)$ is an ${\cal F}^{(N)}_{\underline{t}N}$-martingale even if $\phi(x)={\bf 1}_{K}(x)$ for Borel measurable set $K$. Then, \begin{align}
\left\langle 	M^{(e,N)}(K^c)			\right\rangle_{\underline{t}}=\frac{1}{N}\sum_{\underline{s}< \underline{t}}\sum_{x\in K^{c}N^{\frac{1}{2}}}\frac{\left(\beta B^{(N)}_{\underline{s}N,x}\right)^2}{N^{\frac{3}{2}}}=2\beta^2{(1+{\cal O}(N^{-\frac{1}{2}}))}\mu_{t}(K^c)\notag
\end{align}
is an increasing process. Thus, we have that \begin{align}
P\left(	\sup_{t\leq T}\mu_{t}(K^c)>\ve	\right)&\leq P\left(	3\sup_{t \leq T}\left\langle M^{(e,N)}(K^c)\right\rangle_{\underline{t}}>\ve			\right)\notag\\
&\leq \ve^{-1}E\left[	\frac{3}{N}\sum_{\underline{s}< T}\sum_{x\in K^cN^{\frac{1}{2}}}	\frac{\left(\beta B^{(N)}_{\underline{s}N,x}\right)^2}{N^{\frac{3}{2}}}		\right]\notag\\
&\leq \ve^{-1}C\sum_{\underline{s}<T}\sum_{x\in K^cN^{\frac{1}{2}}}\frac{\beta^2(s\vee 1)^{2}}{N\sqrt{\underline{s}}}P_Y\left(Y_{\underline{s}N}=x\right)\notag\\
&\leq \ve^{-1}C\beta^2\sqrt{T}\left(\sup_{\underline{s}<T}P_Y\left(	Y_{\underline{s}N}\in K^cN^{\frac{1}{2}}		\right)\right)\notag\\
&\leq \ve\notag,
\end{align} 
by taking $K^{c}$ as a compact set in $\mathbb{R}$ such that $C\beta^2\sqrt{K}\sup_{\underline{s}<T}P_Y\left(	Y_{\underline{s}N}\in K^cN^{\frac{1}{2}}		\right)\leq \ve^2$,
where we used Lemma \ref{lem13} in the third inequality.
\end{proof}

In the rest of this section, we will prove Lemma \ref{lem15}.


For $\psi \in C_{b}^{1,2}([0,\infty)\times \mathbb{R}, \mathbb{R})$,  we define \begin{align}
{X}^{(N)}_{{t}}\left(\psi_{{t}}\right)=\sum_{\x\sim \underline{t}}		\frac{\psi\left({t},\frac{Y^{\x}_{\underline{t}N}}{N^{\frac{1}{2}}}\right)}{N},\label{tx}
\end{align}
where $\psi_{t}(x)=\psi(t,x)$. Also, we have the following equation \begin{align}
&{X}^{(N)}_{\underline{t}+\frac{1}{N}}\left(		\psi_{\underline{t}+\frac{1}{N}}	\right)-{X}^{(N)}_{\underline{t}}\left(		\psi_{\underline{t}}			\right)\notag\\
&=\sum_{\x\sim \underline{t}}\frac{\psi\left(\underline{t}+\frac{1}{N},\frac{Y^{\x}_{\underline{t}N+1}}{N^{\frac{1}{2}}}\right)}{N}\left(			V^{\x}-1-\frac{\beta\xi\left(	\underline{t}N,Y^{\x}_{\underline{t}N}	\right)}{N^{\frac{1}{4}}}		\right)\notag\\
&\hspace{2em}+\sum_{\x\sim \underline{t}}\frac{\psi\left(\underline{t}+\frac{1}{N},\frac{Y^{\x}_{\underline{t}N+1}}{N^{\frac{1}{2}}}\right)}{N}\frac{\beta\xi\left(\underline{t}N,Y^{\x}_{\underline{t}N}\right)}{N^{\frac{1}{4}}}\notag\\
&\hspace{2em}+\sum_{\x\sim \underline{t}}\frac{2\psi\left(	\underline{t}+\frac{1}{N},\frac{Y^{\x}_{\underline{t}N+1}}{N^{\frac{1}{2}}}	\right)-\psi\left(\underline{t}+\frac{1}{N},\frac{Y^{\x}_{\underline{t}N}+1}{N^{\frac{1}{2}}}\right)-\psi\left(\underline{t}+1/N,\frac{Y^{\x}_{\underline{t}N}-1}{N^{\frac{1}{2}}}\right)}{2N}
\notag\\
&\hspace{2em}+\sum_{\x\sim\underline{t}}\frac{\psi\left(\underline{t}+\frac{1}{N},\frac{Y^{\x}_{\underline{t}N}+1}{N^{\frac{1}{2}}}\right)+\psi\left(\underline{t}+\frac{1}{N},\frac{Y^{\x}_{\underline{t}N}-1}{N^{\frac{1}{2}}}\right)-2\psi\left(\underline{t},\frac{Y^{\x}_{\underline{t}N}}{N^{\frac{1}{2}}}\right)}{2N}	
\notag\\
&=:\Delta M^{(b,N)}_{\underline{t}+\frac{1}{N}}(\psi_{\underline{t}+\frac{1}{N}})+\Delta M^{(e,N)}_{\underline{t}+\frac{1}{N}}(\psi_{\underline{t}+\frac{1}{N}})\notag\\
&\hspace{2em}+\Delta M^{(s,N)}_{\underline{t}+\frac{1}{N}}(\psi_{\underline{t}+\frac{1}{N}})+\Delta C^{(N)}_{\underline{t}+\frac{1}{N}}(\psi_{\underline{t}+\frac{1}{N}}).\notag
\end{align}

For $i=b,e,s$, $M^{(i,N)}_{t}(\psi_{t})$ which are the sums of $\Delta M^{(i,N)}_{t}(\psi_{t})$ up to $t$ are martingales with respect to ${\cal F}_{\underline{t}N}^{(N)}$ as well as $M^{(i,N)}_{\cdot}(\phi)$ are.

We take $\psi$ as the shift of $\frac{1}{\sqrt{2\pi t}}\exp\left(	-\frac{x^2}{2t}		\right)$;\begin{align}
\psi^{x}_{t}(y)=\frac{1}{\sqrt{2\pi t}}\exp\left(-\frac{(y-x)^{2}}{2t}\right).\notag
\end{align}

Then, we have that for $\ve,\ve'>0$ and $t\geq\eta>0$\begin{align}
&E\left[	\left({X}^{(N)}_{t}\left(\psi^{x}_{\varepsilon}\right)-{X}^{(N)}_{t}\left(		\psi^{x}_{\varepsilon'}	\right)	\right)^{2}	\right]	\notag\\
&\hspace{2em}\leq 		\sum_{\underline{s}\leq \underline{t}}E\left[\left(\Delta M^{(b,N)}_{\underline{s}}\left(\psi^{x}_{	{t}+\ve-\underline{s}}-		\psi^{x}_{t+\ve'-\underline{s}}	\right)\right)^{2}\right]\label{Mb}\tag{Mb}\\
&\hspace{2em}+\sum_{\underline{s}\leq \underline{t}}E\left[\left(\Delta M^{(e,N)}_{\underline{s}}\left(\psi^{x}_{	{t}+\ve-\underline{s}}	-		\psi^{x}_{t+\ve'-\underline{s}}	\right)\right)^{2}\right]\label{Me}\tag{Me}\\
&\hspace{2em}+\sum_{\underline{s}\leq \underline{t}}E\left[\left(\Delta M^{(s,N)}_{\underline{s}}\left(\psi^{x}_{{t}+\ve-\underline{s}}	-		\psi^{x}_{t+\ve'-\underline{s}}	\right)\right)^{2}\right]	\label{Ms}\tag{Ms}\\
&\hspace{2em}+E\left[	\left(\sum_{\underline{s}\leq \underline{t}}	\Delta C^{(N)}_{\underline{s}}\left(\psi^{x}_{{t}+\ve-\underline{s}}-		\psi^{x}_{t+\ve'-\underline{s}}	\right)	\right)^{2}\right]\label{C}\tag{C}\\
&\hspace{2em}+\left(\psi^{x}_{t+\ve}(0)-\psi^{x}_{t+\ve'}(0)\right)^{2}\label{ini}\tag{Initial term}\\
&\hspace{2em}+E\left[\left(\sum_{\x\sim\underline{t}}\frac{\left(\psi^{x}_{\ve}-\psi^{x}_{t+\ve-\underline{t}}-\psi^{x}_{\ve'}+\psi^{x}_{t+\ve'-\underline{t}}\right)\left(\frac{Y^{\x}_{\underline{t}N}}{N^{\frac{1}{2}}}\right)}{2\sqrt{N}}\right)^{2}\right]\label{error}\tag{Error term}
\end{align}

Clearly, for fixed $\ve>0$, $\sup_{y}|\psi_{\ve}^{x}(y)-\psi^{x}_{t+\ve-\underline{t}}(y)|\leq \frac{C(\ve)}{N}$. So (\ref{error}) is bounded above by \begin{align}
E\left[		\left(	X_{\underline{t}}^{(N)}\left(\frac{C(\ve)+C(\ve')}{N}\right)\right)^{2}	\right]\to 0,\ \ \text{as }N\to\infty.\notag
\end{align}
Also, \begin{align}
(\text{\ref{ini}})\leq (\ve-\ve')^{2}\left((t+\ve)\wedge (t+\ve')\right)^{-3},\notag
\end{align}
where we have used \cite[Lemma III 4.5 (a)]{Per}, that is for $0\leq \delta\leq p$,\begin{align}
|\psi_{t+\ve}^{x}(y)-\psi^{x}_{t}(y)|^p\leq \left(\ve t^{-3/2}\right)^{\delta}\left(	\left(\psi^{x}_{t+\ve}(y)\right)^{p-\delta}+\left(\psi^x_{t}(y)\right)^{p-\delta}		\right)\label{normal}
\end{align}
for all $x,y\in\mathbb{R}$, $t>0$, and $\ve>0$.

\begin{lem}\label{lem7}
For $\ve,\ve'>0$ and $t\geq \eta>0$, \begin{align}
\lim_{N\to\infty}E\left[	\left(			\sum_{\underline{s}\leq \underline{t}}\Delta C_{\underline{s}}^{(N)}\left(		\psi^{x}_{t+\ve-\underline{s}}-\psi^{x}_{t+\ve'-\underline{s}}		\right)	\right)	^{2}		\right]=0.\notag
\end{align}
\end{lem}
\begin{proof}
\begin{align}
&\Delta C_{\underline{s}}^{(N)}\left(\psi^{x}_{t+\ve-\underline{s}}\right)\notag\\
&=\sum_{\x\sim\underline{s}}\frac{\psi^{x}_{t+\ve-\underline{s}-\frac{1}{N}}\left(\frac{Y^{\x}_{\underline{t}N}+1}{N^{\frac{1}{2}}}\right)+\psi^{x}_{t+\ve-\underline{s}-\frac{1}{N}}\left(\frac{Y^{\x}_{\underline{t}N}-1}{N^{\frac{1}{2}}}\right)-\psi^{x}_{t+\ve-\underline{s}}\left(\frac{Y^{\x}_{\underline{t}N}+1}{N^{\frac{1}{2}}}\right)-\psi^{x}_{t+\ve-\underline{s}}\left(\frac{Y^{\x}_{\underline{t}N}-1}{N^{\frac{1}{2}}}\right)}{2N}
\notag\\
&+\sum_{\x\sim\underline{s}}		\frac{\psi^{x}_{t+\ve-\underline{s}}\left(	\frac{Y^{\x}_{\underline{t}N}+1}{N^{\frac{1}{2}}}\right)+\psi^{x}_{t+\ve-\underline{s}}\left(	\frac{Y^{\x}_{\underline{t}N}-1}{N^{\frac{1}{2}}}		\right)-2\psi^x_{t+\ve-\underline{s}}\left(	\frac{Y^{\x}_{\underline{t}N}}{N^{\frac{1}{2}}}	\right)}{2N}
\notag\\
\leq& \sum_{\x\sim\underline{s}}\frac{1}{N^{2}}\left(\left.{\frac{\partial \psi^{x}\left(t+\ve-{s},\frac{Y^{\x}_{\underline{s}N}}{N^{\frac{1}{2}}}\right)}{\partial s}}\right|_{s=\underline{s}}+{\cal O}(N^{-\frac{1}{2}})\right)\notag\\
&+\sum_{\x\sim\underline{s}}\frac{1}{N^{2}}\left(		\left.	\frac{\partial^{2} \psi^{x}(t+\ve-\underline{s},y)}{2\partial y^{2}}\right|_{y=\frac{Y^{\x}_{\underline{s}N}}{N^{\frac{1}{2}}}}+{\cal O}(N^{-\frac{1}{2}})	\right)\notag
\end{align}
Since $\frac{\partial \psi^{x}(t+\ve-s,y)}{\partial s}+\frac{\partial^{2}\psi^{x}(t+\ve-s,y)}{2\partial y^{}2}=0$, the last equation is bounded above by\begin{align}
\left|	\Delta C_{\underline{s}}^{(N)}\left(		\psi^{x}_{t+\ve-\underline{s}}	\right)	\right|\leq C(\ve,\eta)\frac{X_{\underline{s}}^{(N)}(1)}{N^{\frac{3}{2}}}.\notag
\end{align}
Thus, \begin{align}
E\left[	\left(		\sum_{\underline{s}\leq \underline{t}}\Delta C_{\underline{s}}^{(N)}\left(		\psi^{x}_{t+\ve-\underline{s}}-\psi^{x}_{t+\ve'-\underline{s}}		\right)		\right)^{2}			\right]&\leq E\left[		\left(C(\ve,\eta)+C(\ve',\eta)\right)^{2}\sup_{\underline{s}\leq\underline{t}}\left(\frac{X^{(N)}_{\underline{s}}(1)	}{N^{\frac{1}{2}}}\right)^{2}\right]\notag\\
&\to 0 \ \ \text{as }N\to \infty.\notag 
\end{align}
Indeed, for each $N$,  $X^{(N)}_{\underline{s}}(1)$ is a martingale so that by $L^{2}$-maximum inequality and by Lemma \ref{lem12}, \begin{align}
\sup_{N}E\left[		\sup_{\underline{s}\leq \underline{t}}\left(X^{(N)}_{\underline{s}}(1)	\right)^{2}	\right] \leq 4\sup_{N}E\left[\left\langle		X^{(N)}(1)\right\rangle_{\underline{t}}\right]<\infty.\notag
\end{align}
\end{proof}

Thus, we have by Fatou's lemma that \begin{align}
&E\left[	\left(		{X}_{t}(\psi^{x}_{\ve})-X_{t}(\psi^{x}_{\ve'})		\right)^{2}			\right]			\notag\\
&\hspace{2em}\leq (\ve-\ve')^2(t+\ve\wedge \ve')^{-3}X_0(1)^2+\liminf _{N\to\infty}\left((\ref{Mb})+(\ref{Me})+(\ref{Ms})\right).\notag
\end{align}

Hereafter, we will see the right hand side .

\begin{lem}\label{lem8}
Suppose $\ve>\ve'>0$, $t\geq \eta>0$, and $0<\delta<\frac{1}{2}$. Then, for any $x\in \mathbb{R}$\begin{align}
\liminf_{N\to\infty}\ (\text{\ref{Mb}})\leq C_{\delta}(\ve-\ve')^{\delta}(t+\ve')^{-\delta}.\notag
\end{align} 
\end{lem}

\begin{proof}
By Lemma \ref{lem4}, we have that for $\ve>\ve'>0$, for $t\geq \eta>0$, and for $0<\delta<\frac{1}{2}$
\begin{align}
&(\text{\ref{Mb}})\notag\\
&=\left(1-\frac{1}{N^{\frac{1}{2}}}\right)E\left[		\sum_{\underline{s}\leq \underline{t}}\sum_{z\in\mathbb{Z}}\frac{\left(		\psi^{x}_{t+\ve-\underline{s}}\left(	\frac{z}{N^{\frac{1}{2}}}		\right)-\psi^{x}_{t+\ve'-\underline{s}}\left(		\frac{z}{N^{\frac{1}{2}}}		\right)					\right)^2}{N^2}	B_{\underline{s}N,z}^{(N)}				\right]\notag\\
&\leq E_Y\left[	\sum_{\underline{s}\leq \underline{t}}\frac{\left(\psi^{x}_{t+\ve-\underline{s}}\left(\frac{Y_{\underline{s}N}}{N^{\frac{1}{2}}}\right)	-\psi^{x}_{t+\ve'-\underline{s}}\left(\frac{Y_{\underline{s}N}}{N^{\frac{1}{2}}}\right)		\right)^{2}}{N}			\right],
\notag
\intertext{and it follows from (\ref{normal}) that } 
&\leq \int_{0}^{\underline{t}}E_Y\left[	\left(\frac{\ve-\ve'}{(t+\ve'-s)^{\frac{3}{2}}}	\right)^{\delta}	\left(		\left(\psi_{t+\ve-s}^{x}\left(\frac{Y_{\underline{s}N}}{N^{\frac{1}{2}}}\right)	\right)^{2-\delta}+\left(\psi_{t+\ve'-s}^{x}\left(\frac{Y_{\underline{s}N}}{N^{\frac{1}{2}}}\right)	\right)^{2-\delta}	\right)\right]ds \notag.
\end{align}
Thus,  we have from invariance principle that 
\begin{align}
&\liminf_{N\to\infty}(\text{\ref{Mb}})\notag\\
&\leq \int_{0}^{{t}}\int_{\mathbb{R}}	\left(\frac{\ve-\ve'}{(t+\ve'-s)^{\frac{3}{2}}}	\right)^{\delta}	\left(		\left(\psi_{t+\ve-s}^{x}\left(y\right)	\right)^{2-\delta}+\left(\psi_{t+\ve'-s}^{x}\left(y\right)	\right)^{2-\delta}	\right)
\psi_{s}^{0}(y)dyds
\notag\\
&\leq (\ve-\ve')^{\delta}\int_{0}^{t}(t+\ve'-s)^{-\frac{3\delta}{2}}(2-\delta)^{-\frac{1}{2}}\left((t+\ve-s)^{\frac{\delta-1}{2}}\left(\frac{2-\delta}{t+\ve+(1-\delta)s}\right)^{\frac{1}{2}}\right)ds\notag\\
&+(\ve-\ve')^{\delta}\int_{0}^{t}(t+\ve'-s)^{-\frac{3\delta}{2}}(2-\delta)^{-\frac{1}{2}}\left((t+\ve'-s)^{\frac{\delta-1}{2}}\left(\frac{2-\delta}{t+\ve'+(1-\delta)s}\right)^{\frac{1}{2}}\right)ds\ \ \notag\\
&\leq C_{\delta}(\ve-\ve')^{\delta}(t+\ve')^{-\frac{1}{2}}\int_{0}^{t}(t+\ve'-s)^{-\frac{1}{2}-\delta}ds\leq C_{\delta}(\ve-\ve')^{\delta}(t+\ve')^{-\frac{1}{2}}(t+\ve')^{\frac{1}{2}-\delta}\notag,
\end{align}
where we have used the fact that $\int_{\mathbb{R}}\psi^{x}_{s}(y)\psi^{0}_{t}(y)dy=\psi^{0}_{t+s}(x)$ in the second inequality.

\end{proof}


\begin{lem}\label{lem9}
For all $x\in\mathbb{R}$, $\ve>\ve'>0$, and $t\geq \eta>0$, we have \begin{align}
\lim_{N\to\infty}\ (\text{Ms})=0.\notag
\end{align}\end{lem}
\begin{proof}
The proof is the same as the proof of Lemma \ref{lem3}.
\end{proof}


\begin{lem}\label{lem10}
Suppose $\ve>\ve'>0$, $t\geq \eta>0$, and $0<\delta<\frac{1}{2}$. Then, for any $x\in \mathbb{R}$
\begin{align}
\liminf_{N\to\infty}\text{(\ref{Me})}\leq C(\delta)\beta^2(t\vee 1)^2(t+\ve')^{-\frac{1}{2}-\delta}(\ve-\ve')^{\delta}.\notag
\end{align}
\end{lem}


\begin{proof}
By Lemma \ref{lemma14}, we have that 
\begin{align}
&\text{(\ref{Me})}\notag\\
&\leq \beta^2E\left[	\sum_{\underline{s}\leq \underline{t}}\sum_{z\in\mathbb{Z}}\frac{\left(		\psi^{x}_{t+\ve-\underline{s}}\left(	\frac{z}{N^{\frac{1}{2}}}		\right)-\psi^{x}_{t+\ve'-\underline{s}}\left(	\frac{z}{N^{\frac{1}{2}}}		\right)		\right)^{2}}{N}		\frac{\left(B^{(N)}_{\underline{s}N,z}\right)^{2}}{N^{\frac{3}{2}}}	\right]\notag\\
&\leq \beta^2\sum_{\underline{s}\leq \underline{t}}\sum_{z\in\mathbb{Z}}\frac{C(\underline{s}\vee 1)^2}{N\sqrt{\underline{s}}}		\left(	\psi_{t+\ve-\underline{s}}^{x}\left(		\frac{z}{N^{\frac{1}{2}}}	\right)	-	\psi_{t+\ve'-\underline{s}}^{x}\left(		\frac{z}{N^{\frac{1}{2}}}	\right)	
\right)^{2}	P\left(		Y_{\underline{s}N}=z		\right)\notag\\
&\leq C\beta^2(t\vee 1)^2\int_{0}^{t}\sum_{z\in\mathbb{Z}}\frac{1}{\sqrt{s}}\left(\psi_{t+\ve-{s}}^{x}\left(\frac{z}{N^{\frac{1}{2}}}\right)-\psi^{x}_{t+\ve'-{s}}\left(\frac{z}{N^{\frac{1}{2}}}\right)\right)^{2}P\left(Y_{\underline{s}N}=z\right)ds,\notag
\end{align}
where we have used Lemma \ref{lem13} in the third inequality. 
Let $0<\eta'<t$. Then, we obtain by the similar argument in the proof of Lemma \ref{lem8} that
\begin{align}
&\liminf_{N\to \infty}(\text{\ref{Mb}})\notag\\
&\leq C\beta^2(t\vee 1)^2\left(\int_{\eta'}^{t}\int_{\mathbb{R}}\frac{1}{\sqrt{s}}\left(\psi_{t+\ve-{s}}^{x}\left(y\right)-\psi^{x}_{t+\ve'-{s}}\left(y\right)\right)^{2}\psi^0_{s}(y)dyds\right.\notag\\
&\hspace{7em}+\left.\int_{0}^{\eta'}\frac{\sup_{y}\left(\psi^x_{t+\ve-s}(y)-\psi^x_{t+\ve'-s}(y)\right)^2}{\sqrt{s}}ds\right)\notag\\
&\leq C\beta^2(t\vee 1)^2\int_{\eta'}^{t}\int_{\mathbb{R}}\left(\frac{\ve-\ve'}{(t+\ve'-s)^{\frac{3}{2}}}\right)^{\delta}\left(		\left(\psi^{x}_{t+\ve-s}(y)\right)	^{2-\delta}+\left(\psi^{x}_{t+\ve'-s}(y)\right)^{2-\delta}	\right)\frac{\psi^{0}_{s}(y)}{\sqrt{s}}dyds\notag\\
&+C\beta^2(t\vee 1)^2\int_0^{\eta'}\left(\frac{\ve-\ve'}{(t+\ve'-s)^{\frac{3}{2}}}\right)^{\delta}s^{-\frac{1}{2}}\left((t+\ve-s)^{\frac{2-\delta}{2}}+(t+\ve'-s)^{\frac{2-\delta}{2}}\right)ds\notag\\
&\leq C(\delta)\beta^2\frac{(t\vee 1)^2}{(t+\ve')^{\frac{1}{2}}}(\ve-\ve')^{\delta}\int_{0}^{t+\ve'}s^{-\frac{1}{2}}(t+\ve'-s)^{-\frac{1}{2}-\delta}ds\notag\\
&+C(\delta)\beta^2(t\vee 1)^2(\ve-\ve')^{\delta}\int_0^{\eta'}s^{-\frac{1}{2}}\frac{(t+\ve'-s)^{\frac{2-\delta}{2}}+(t+\ve-s)^{\frac{2-\delta}{2}}}{(t+\ve'-s)^{\delta}}ds\notag\\
&\leq C(\delta)\beta^2(t\vee 1)^2(\ve-\ve')^{\delta}\left(\left(t+\ve'\right)^{-\frac{1}{2}-\delta}B\left(\frac{1}{2},\frac{1}{2}-\delta\right)+\eta'^{\frac{1}{2}}(t+\ve)^{\frac{2-\delta}{2}}(t+\ve'-\eta')^{-\delta}\right).\notag
\end{align}
Since $\eta'>0$ is arbitrary, we have that \begin{align*}
\liminf_{N_k\to\infty}\text{(\ref{Me})}\leq C(\delta)\beta^2(t\vee 1)^2(t+\ve')^{-\frac{1}{2}-\delta}(\ve-\ve')^{\delta}.
\end{align*}
\end{proof}

Thus, we have that \begin{align}
\lim_{\ve,\ve' \to 0}\sup_{x\in\mathbb{R},t\geq \eta}E\left[\left(X_{t}(\psi_{\ve}^{x})-X_{t}(\psi_{\ve'}^{x})\right)^{2}\right]=0,\ \ \text{for any }\eta>0.\notag
\end{align}
By Skorohod representation theorem, we may assume that $X^{(N_{k})}$ and $X$ are defined on a common probability space and $X^{(N_{k})}\to X$ in $D([0,\infty),{\cal M}_{F}(\mathbb{R}))$ a.s.. Then, from the above arguments, we have that \begin{align}
X_{t}(\psi^{x}_{\ve})=X_{0}(\psi^{x}_{t+\ve})+\tilde{M}_{t}(\psi^{x}_{t+\ve-\cdot})\label{limit}
\end{align}
for a certain continuous $L^{2}$-bounded martingale $\tilde{M}_{t}(\psi^{x}_{t+\ve-\cdot})$, where the martingale property of $\tilde{M}_t\left(	\psi^{x}_{t+\ve-\cdot}		\right)$ is obtained by the same argument as the proof  of Lemma \ref{lem0}. Also, we take $L^{2}$-limit in (\ref{limit}) as $\ve\to 0$ and choose $\ve_{n}\to 0$ so that for any $t$ and $x\in \mathbb{R}$,\begin{align}
\lim_{n\to\infty} X_{t}(\psi^{x}_{\ve_{n}})=X_{0}(\psi^{x}_{t})+\tilde{M}_{t}(\psi^{x}_{t-\cdot})\ \ \text{a.s. and in }L^{2}.\label{dens}
\end{align}
We define $u(t,x)=\lim_{\ve_n\to 0} X_{t}(\psi^{x}_{\ve_{n}})$ for all $t>0$, $x\in \mathbb{R}$. 
Standard differential theory shows that for each $t>0$ with probability $1$,\begin{align}
X_{t}(dx)=u(t,x)dx+X^{s}_{t}(dx),\notag
\end{align}
where $X^{s}_{t}$ is a random measure such that $X^{s}_{t}(dx) \perp dx$. Also, (\ref{dens}) implies that \begin{align}
E\left[\int_{\mathbb{R}}u(t,x)dx\right]=\int_{\mathbb{R}}X_{0}(\psi^{x}_{t})dx=1=E\left[X_{t}(1)\right].\notag
\end{align}
Thus, $E\left[X^{s}_{t}(1)\right]=0$ and \begin{align}
X_{t}(dx)=u(t,x)dx,\ \ \text{a.s.\,for all }t>0.\notag
\end{align}

Therefore, we complete the proof of Lemma \ref{lem15} and also of the existence of Theorem \ref{thm1}.

\subsection{Weak uniqueness of the limit point process}
In the end of this section, we will prove the weak uniqueness of the limit point process $X$.

The main idea is to prove the existence of the ``dual process" $\{Y_t:t\geq 0\}$, which is $C_b^+(\mathbb{R})$-valued process independent of $X$ satisfying \begin{align}
E\left[	\exp\left(-\langle		X_t,\phi	\rangle	\right)	\right]=E\left[		\exp\left(	-\langle		X_0,Y_t			\rangle		\right)	\right]\label{dual}
\end{align}
for each $\phi\in C(\mathbb{R})$, where $\langle	\nu,\phi		\rangle=\int_\mathbb{R}\phi(x)\nu(dx)$ for $\nu\in{\cal M}_F(\mathbb{R})$ and $\phi\in C_b(\mathbb{R})$. The reader should be careful not to  confuse the notation of quadratic variation of martingale. Also, we will identify $v\in L^1_+(\mathbb{R})$ as a finite measure on $\mathbb{R}$ by $v(x)dx$.

Then, it is clear that the closure of $C_{rap}^+(\mathbb{R})$ under convergence with bounded pointwisely is the set of nonnegative bounded measurable functions. Thus, it is enough to show (\ref{dual}) for all $\phi\in C^+_{rap}(\mathbb{R})$ from Lemma II.\,5.9 in \cite{Per}.

In our case, the dual process is a solution to the martingale problem:\begin{align}
\begin{cases}
\text{For all }\psi\in D(\Delta),\\
\quad \tilde{Z}_t(\psi)=\langle	Y_t,\psi		\rangle-\langle		Y_0,\psi	\rangle-\frac{\gamma}{2}\int_0^t	\langle Y_s^2,\psi\rangle ds	-\int_0^t\left\langle	Y_s,\frac{1}{2}\Delta \psi		\right\rangle ds\\
\text{is an }{\cal F}_t^Y\text{-continuous square integrable martingale and }\\
\quad \langle		\tilde{Z}(\psi)	\rangle_t=2\beta^2\int_0^t	\langle Y_s^2,\psi^2\rangle ds.	
\end{cases}\label{dual2}
\end{align}

A solution to such martingale problem is a solution of the nonlinear stochastic heat equation:\begin{align}
\frac{\partial }{\partial t}Y_t(x)=\frac{1}{2}\Delta Y_t(x)-\frac{\gamma}{2}Y_t(x)^2+\sqrt{2\beta}Y_t(x)\dot{\tilde{W}}(t,x). \label{nonlinear}
\end{align}
The existence of the nonnegative solution of (\ref{nonlinear}) for the case where $Y_0\in C_{rap}^+(\mathbb{R})$ has been proved in \cite{Nak2}.

We will see that solutions to (\ref{nonlinear}) satisfies (\ref{dual}).

\begin{proof}[Proof of the uniqueness]
Let $X$ be a solution to the martingale problem obtained as  a limit point of $X^{(N)}$ and also, we denote by $X_t(x)$ its density.
Let $Y$ be a solution to the martingale problem (\ref{dual2})  constructed on the same probability space as $\{X,X^{(N)}\}$ and independent of them.

We denote by $\nu^\ve$ the convolution of $\nu \in {\cal M}_F(\mathbb{R})$ and 
 $p_{\ve}(x)=\frac{1}{\sqrt{2\pi\ve}}\exp\left(	-\frac{x^2}{2\ve}\right)$ for $\ve>0$ and $x\in\mathbb{R}$. Then, it is clear that $Y^\ve\in C_{b,+}^2(\mathbb{R})$.
We have by It\^{o}'s lemma that for fixed $u\in[0,t]$\begin{align*}
&\exp\left(	-\left\langle 	X_t,Y_{u}^\ve	\right\rangle			\right)\\
&\hspace{1em}-\int_0^t	\exp\left(	-\left\langle X_s,Y_u^\ve	\right\rangle	\right)\left(\left\langle	X_s,\frac{1}{2}\Delta Y_u^\ve		\right\rangle	-\frac{\gamma}{2} \left\langle X_s,	\left(Y_u^\ve\right)^2\right\rangle -{\beta^2}		\left\langle	X^2_s,\left(Y_u^{\ve}\right)^2		 \right\rangle		\right)	ds
\end{align*}
is an ${\cal F}_t^X$-martingale and\begin{align*}
&\exp\left(	-\left\langle 	X_u,Y_{t}^\ve	\right\rangle			\right)\\
&\hspace{1em}-\int_0^t	\exp\left(	-\left\langle X_u,Y_s^\ve	\right\rangle	\right)\left(\left\langle	X_u,\frac{1}{2}\Delta Y_s^\ve		\right\rangle	-\frac{\gamma}{2} \left\langle X_u^\ve,	Y_s^2\right\rangle -{\beta^2}		\left\langle	\left(X_u^\ve\right)^2, Y_s^2		 \right\rangle		\right)	ds
\end{align*}
is an ${\cal F}_t^Y$-martingale.
We define for $0\leq s,t\leq T<\infty$ and $\ve\geq 0$\begin{align*}
g(t,s,\ve)=E\left[		\exp\left(	-\left\langle 	X_t,Y_s^\ve	\right\rangle				\right)		\right].
\end{align*}
Then, we have that by Fubini's theorem\begin{align*}
&\int_0^T(g(s,0,\ve)-g(0,s,\ve))ds\\
&=\int_0^T\left(g(T-s,s,\ve )-g(0,s,\ve)-g(s,T-s,\ve)+g(s,0,\ve)\right)ds\\
&=E\left[\int_0^T	\int_0^{T-s}\exp\left(	-\left\langle		X_{u},Y_s^\ve		\right\rangle		\right)\left\{	\frac{\gamma}{2}\left\langle	X_{u	},	\left(	\left(Y_s^2\right)^\ve-\left(Y_s^\ve\right)^2		\right)\right\rangle		\right.\right.\\
&\hspace{15em}\left.\left.		+\beta^2\left(\left\langle	\left(X_{u}^\ve\right)^2,Y_s^2			\right\rangle	-\left\langle		X_u,\left(Y_s^\ve\right)^2		\right\rangle	\right)	\right\}duds\right].
\end{align*}
We will show that the right hand side converges to $0$ as $\ve\to 0$. Then, $g(s,0,\ve)\to g(s,0,0)$ and $g(0,s,\ve)\to g(0,s,0)$  as $\ve\to 0$. Also, since $g(s,0,0)$ and $g(0,s,0)$ are continuous at $s=T$, differentiating both sides of the above equation at $s=T$, we get $g(T,0,0)=g(0,T,0)$.
Since $\exp\left(		-\langle  X_u,Y_s^\ve\rangle	\right)\leq 1$, it is enough to show that\begin{align*}
&\int_0^T\int_0^{T-u}E\left[	\left\langle	X_u,	\left|	\left(Y_s^\ve\right)^2-Y_s^2\right|	\right\rangle		\right]dsdu\to 0\\
&\int_0^T\int_0^{T-u}E\left[	\left\langle	X_u,	\left|	\left(Y_s^2\right)^\ve-Y_s^2\right|	\right\rangle		\right]dsdu\to 0\\
&\int_0^T\int_0^{T-u}E\left[		\left\langle	\left|\left(X_{u}^\ve\right)^2-X_u^2\right|,Y_s^2	\right\rangle	\right]dsdu\to 0\\
&\int_0^T\int_0^{T-u}E\left[		\left\langle	X_{u}^2,\left|	\left(Y_s^\ve\right)^2	-Y_s^2	\right|	\right\rangle	\right]dsdu\to 0\text{, as }\ve\to0.
\end{align*}
Now, we will show them.
It is clear that
\begin{align*}
&\int_0^T\int_0^{T-u}E\left[	\left\langle	X_u,	\left|	\left(Y_s^\ve\right)^2-Y_s^2\right|	\right\rangle		\right]dsdu\\
&=\int_0^T\int_0^{T-u}\left\langle	E_X[X_u], E_Y\left[\left|	\left(Y_s^\ve\right)^2-Y_s^2\right|\right]	\right\rangle dsdu\\
&=\int_0^T \left\langle  	E_X[X_u],	\left(\int_0^{T-u}E_Y\left[\left(	Y_s^\ve-Y_s\right)^2\right]du	\right)^{1/2}\left(\int_0^{T-u}E_Y\left[\left(	Y_s^\ve+Y_s\right)^2\right]du	\right)^{1/2}	\right\rangle ds.
\end{align*}
From the proof of  Proposition 3.7 in \cite{Nak2}, we have that\begin{align*}
&E_Y\left[\left(	Y_s^\ve(x)-Y_s(x)\right)^2\right]\\
&\leq \int_\mathbb{R}E_Y\left[\left(Y_s(x+y)-Y_s(x)\right)^2\right]p_{\ve}(y)dy\\
&\leq 	C(\gamma,\beta,T)\int_\mathbb{R} E\left[\left(\int_\mathbb{R}\left(p_s(x+y+z)-p_s(x+z)\right)Y_0(z)dz\right)^2	\right.\\
&\hspace{8em}\left.	+\left(\int_0^s\int_\mathbb{R}\left(p_{s-u}(x+y+z)-p_{s-u}(x+z)\right)Y_u^2(z)dzdu\right)^2\right.\\
&\hspace{8em}\left.+\int_0^s\int_\mathbb{R}\left(p_{s-u}(x+y+z)-p_{s-u}(x+z)\right)^2Y^2_u(z)dzdu			\right]p_\ve(y)dy	.
\end{align*}

Then, it follows from H\"{o}lder's inequality that \begin{align*}
&\int_0^{T-u}\left(\int_\mathbb{R}\left(p_s(x+y+z)-p_s(x+z)\right)Y_0(z)dz\right)^2ds\\
&\hspace{3em}\leq \int_0^{T-u}\left(\int_\mathbb{R}\left(p_s(x+y+z)-p_s(x+z)\right)^2dz\int_\mathbb{R}Y_0^2(z)dz\right)ds\\
&\hspace{3em}\leq  C(Y_0)|y|,\\
&\int_0^{T-u}E\left[\left(\int_0^s\int_\mathbb{R}\left(p_{s-r}(x+y+z)-p_{s-r}(x+z)\right)Y_r^2(z)dzdr\right)^2\right]ds\\
&\hspace{3em}\leq \int_0^{T-u}\left(	E\left[		\int_0^s\int_\mathbb{R}Y_r^2	(z)dzdr\right]\right)\\
&\hspace{7em}\times\left(\int_0^s\int_\mathbb{R}	\left(p_{s-r}(x+y+z)-p_{s-r}(x+z)\right)^2dzdr\right)ds\\
&\hspace{3em}\leq C(Y_0,T)|y|,
\intertext{and}
&\int_0^{T-u}	\int_0^s\int_\mathbb{R}\left(p_{s-r}(x+y+z)-p_{s-r}(x+z)\right)^2E\left[Y_r^2(z)\right]dzduds\\
&\leq \int_0^{T-u}\left(	\int_0^s\int_\mathbb{R}\left(p_{s-r}(x+y+z)-p_{s-r}(x+z)\right)^2dzdr\right)^{1/2}\\
&\hspace{5em}		\times \left(\int_0^{s}\int_\mathbb{R}\left(p_{s-r}(x+y+z)-p_{s-r}(x+z)\right)^2E\left[Y_r^2(z)\right]dzdu\right)^{1/2}ds	\\
&\leq C\int_0^{T-u}|y|^{1/2}\left(\int_0^s	\int_{\mathbb{R}}\left(p_{s-r}(x+y+z)^{8/3}+p_{s-r}(x+z)^{8/3}	\right)dzdr	\right)^{3/8}\\
&\hspace{17em\times }\left(	\int_0^s\int_\mathbb{R} E\left[	Y_r^{2}(z)	\right]	dzdr	\right)^{1/8}ds\\
&\leq C(Y_0,T)|y|^{1/2}.
\end{align*}
Thus, we have that \begin{align*}
\left(\int_0^{T-u}E_Y\left[\left(	Y_s^\ve-Y_s\right)^2\right]du	\right)^{1/2}\leq C(Y_0,T)(|\ve|^{1/4}+|\ve|^{1/2})^{1/2}.
\end{align*}

Also, we have by the same argument that\begin{align*}
E_Y\left[(Y^\ve_s(x)+Y_s(x))^2			\right]<C(Y_0,T).
\end{align*}
Therefore, we can show that \begin{align*}
\int_0^T\int_0^{T-u}E\left[	\left\langle	X_u,	\left|	\left(Y_s^\ve\right)^2-Y_s^2\right|	\right\rangle		\right]dsdu\to 0.
\end{align*}
Also,  the similar argument implies that \begin{align*}
\int_0^T\int_0^{T-u}E\left[	\left\langle	X_u,	\left|	\left(Y_s^2\right)^\ve-Y_s^2\right|	\right\rangle		\right]dsdu\to 0
\intertext{and }
\int_0^T\int_0^{T-u}E\left[		\left\langle	X_{u}^2,\left|	\left(Y_s^\ve\right)^2	-Y_s^2	\right|	\right\rangle	\right]dsdu\to 0,
\end{align*}
as $\ve \to 0$, where we have used from that \begin{align*}
\int_0^TE\left[	\left\langle X_s^2,1\right\rangle			\right]ds\leq \varliminf_{N\to\infty}E\left[	\left\langle	X^{(N)}(1)		\right\rangle_T		\right]<\infty.
\end{align*}
We will complete the proof by showing that \begin{align*}
\int_0^T\int_0^{T-u}E\left[		\left\langle	\left|\left(X_{u}^\ve\right)^2-X_u^2\right|,Y_s^2	\right\rangle	\right]dsdu\to 0\text{, as }\ve\to0.
\end{align*}
This is true when the case $X_0$ has a rapidly decreasing continuous density.

We have  by Fatou's lemma that for any $u>0$\begin{align*}
&\int_0^{T-u}E\left[		\left\langle	\left|\left(X_{u}^\ve\right)^2-X_u^2\right|,Y_s^2	\right\rangle	\right]ds\\
&\leq \left(\int_0^{T-u}E\left[		\int_\mathbb{R}			\left|X_{u}^\ve+X_u\right|^2Y_s^2(x)	dx	\right]ds\right)^{1/2}\\
&\hspace{2em}\times \left(\int_0^{T-u}E\left[		\int_\mathbb{R}			\left|X_{u}^\ve-X_u\right|^2Y_s^2(x)	dx	\right]ds\right)^{1/2}\\
&\leq\liminf_{\ve'\to 0} \left(\int_0^{T-u}E\left[		\int_\mathbb{R}			\left|X_{u}^\ve+X_u^{\ve'}\right|^2Y_s^2(x)	dx	\right]ds\right)^{1/2}\\
&\hspace{2em}\times \liminf_{\ve'\to 0}\left(\int_0^{T-u}E\left[		\int_\mathbb{R}			\left|X_{u}^\ve-X_u^{\ve'}\right|^2Y_s^2	(x)dx	\right]ds\right)^{1/2}.
\end{align*}
Also, it follows from the construction of $X$ that by Fubini's theorem and Fatou's lemma\begin{align}
&\int_0^{T-u}E\left[		\int_\mathbb{R}			\left|X_{u}^\ve+X_u^{\ve'}\right|^2Y_s^2(x)	dx	\right]ds\notag\\
&=\int_0^{T-u}E\left[		\int_\mathbb{R}			\left|X_{u}(p_{\ve}(x+\cdot)+p_{\ve'}(x+\cdot))\right|^2Y_s^2(x)	dx	\right]ds\notag\\
&\leq \int_0^{T-u}\int_\mathbb{R}\liminf_{N\to \infty} E\left[	\left|X^{(N)}_{u}(p_{\ve}(x+\cdot)+p_{\ve'}(x+\cdot))\right|^2Y_s^2(x)		\right]dxds\notag\\
&\leq  \int_0^{T-u}\int_\mathbb{R}\liminf_{N\to \infty} E\left[	\left|X^{(N)}_{0}(p_{u+{\ve}}(x+\cdot)+p_{u+\ve'}(x+\cdot))\right|^2Y_s^2(x)		\right]dxds\label{First}\\
&\hspace{0em}+\int_0^{T-u}\int_\mathbb{R}\liminf_{N\to \infty} \gamma E\left[	\int_0^uX^{(N)}_{r}\left((p_{u-r+{\ve}}(x+\cdot)+p_{u-r+\ve'}(x+\cdot))^2\right)Y_s^2(x)	dr	\right]dxds\label{second}\\
&\hspace{0em}+\int_0^{T-u}\int_\mathbb{R}\liminf_{N\to \infty} 2\beta^2E\left[	\int_0^{u}\int_\mathbb{R}\left(u^{(N)}(r,z)\right)^2	\left(p_{u-r+\ve}(x+z)+p_{u-r+\ve'}(x+z)\right)^2Y_s^2(x)dz dr	\right]dxds.\label{third}
\end{align}
We will estimate each term.
\begin{align*}
&(\text{\ref{First}})\\
&= \int_0^{T-u}\int_\mathbb{R}\liminf_{N\to \infty} E\left[	\left|X^{(N)}_{0}(p_{u+{\ve}}(x+\cdot)+p_{u+\ve'}(x+\cdot))\right|^2Y_s^2(x)		\right]dxds\\
&\leq \int_0^{T-u}\int_\mathbb{R}\liminf_{N\to \infty} E\left[	X^{(N)}_{0}(1)X_0^{(N)}\left(\left(p_{u+{\ve}}(x+\cdot)+p_{u+\ve'}(x+\cdot)\right)^2\right)Y_s^2(x)		\right]dxds\\
&=\int_0^{T-u}\int_\mathbb{R}X_0(1)X_0\left(\left(p_{u+{\ve}}(x+\cdot)+p_{u+\ve'}(x+\cdot)\right)^2E\left[	Y_s^2(x)		\right]\right)dxds\\
&=\int_0^{T-u}X_0(1)X_0\left(\int_\mathbb{R}\left(p_{u+{\ve}}(x+\cdot)+p_{u+\ve'}(x+\cdot)\right)^2E\left[	Y_s^2(x)		\right]dx\right)ds\\
&\leq \int_0^{T-u}X_0(1)X_0\left(\left(\int_\mathbb{R}\left(p_{u+{\ve}}(x+\cdot)+p_{u+\ve'}(x+\cdot)\right)^{8/3}dx\right)^{3/4}\left(\int_\mathbb{R}E\left[	Y_s^8(x)		\right]dx\right)^{1/4}\right)ds\\
&\leq C(T,X_0,Y_0)\left((u+\ve)^{-5/8}+(u+\ve')^{-5/8}\right).
\end{align*}
It follows by the similar argument to the proof of Lemma \ref{lem8} that
\begin{align*}
&(\text{\ref{second}})\\
&= \gamma \int_0^{T-u}\int_\mathbb{R} E\left[Y_s^2(x)\right]\liminf_{N\to\infty} E\left[	\int_0^uX^{(N)}_{r}\left((p_{u-r+{\ve}}(x+\cdot)+p_{u-r+\ve'}(x+\cdot))^2\right)	dr	\right]dxds\\
&\leq\gamma \int_0^{T-u}\int_\mathbb{R} \left(E\left[Y_s^2(x)\right]    \int_0^{u}\int_\mathbb{R}p_{r}(y)\left(p_{u-r+{\ve}}(x+y)^2+p_{u-r+\ve'}(x+y)^2\right)dy	dr\right)	dxds\\
&\leq C(T,X_0,Y_0).
\end{align*}
Also, it follows from the similar argument to the proof of Lemma \ref{lem10} that\begin{align*}
&(\text{\ref{third}})\\
&=2\beta^2 \int_0^{T-u}\int_\mathbb{R}E\left[Y_s^2(x)\right]\liminf_{N\to \infty} E\left[	\int_0^{u}\int_\mathbb{R}\left(u^{(N)}(r,z)\right)^2	\left(p_{u-r+\ve}(x+z)+p_{u-r+\ve'}(x+z)\right)^2dz dr	\right]dxds\\
&\leq C(T,\beta,X_0)\int_0^{T-u}\int_\mathbb{R}E\left[Y_s^2(x)\right] 	\int_0^{u}\int_\mathbb{R}\frac{1}{\sqrt{r}}p_r(z)	\left(p_{u-r+\ve}(x+z)^2+p_{u-r+\ve'}(x+z)^2\right)dz dr	dxds\\
&\leq C(T,\beta,X_0)\int_0^{T-u}\int_\mathbb{R}\left((u+\ve)^{-1/2}+(u+\ve')^{-1/2}\right)E\left[Y_s^2(x)\right]dxds\\
&\leq C(T,\beta^2,X_0,Y_0)\left((u+\ve)^{-1/2}+(u+\ve')^{-1/2}\right).
\end{align*}
Thus, we have that \begin{align*}
&\left(\liminf_{\ve'\to0}\int_0^{T-u}E\left[		\int_\mathbb{R}			\left|X_{u}^\ve+X_u^{\ve'}\right|^2Y_s^2(x)	dx	\right]ds\right)^{1/2}\\
&\leq C(T,X_0,Y_0,\gamma,\beta)\left(1+u^{-1/2}+u^{-5/8}\right)^{1/2}.
\end{align*}
Also, the similar argument does hold for the term $\dis  \int_0^{T-u}E\left[		\int_\mathbb{R}			\left|X_{u}^\ve-X_u^{\ve'}\right|^2Y_s^2(x)	dx	\right]ds$.
Actually, we have from Lemma \ref{lem8} and Lemma \ref{lem10} and by (\ref{normal}) that for any $0< \delta< \frac{1}{2}$\begin{align*}
 &\int_0^{T-u}E\left[		\int_\mathbb{R}			\left|X_{u}^\ve-X_u^{\ve'}\right|^2Y_s^2(x)	dx	\right]ds\\
 &\leq  \int_0^{T-u}\int_\mathbb{R}\liminf_{N\to \infty} E\left[	\left|X^{(N)}_{0}(p_{u+{\ve}}(x+\cdot)-p_{u+\ve'}(x+\cdot))\right|^2Y_s^2(x)		\right]dxds\\
&\hspace{0em}+\int_0^{T-u}\int_\mathbb{R}\liminf_{N\to \infty} \gamma E\left[	\int_0^uX^{(N)}_{r}\left((p_{u-r+{\ve}}(x+\cdot)-p_{u-r+\ve'}(x+\cdot))^2\right)Y_s^2(x)	dr	\right]dxds\\
&\hspace{0em}+\int_0^{T-u}\int_\mathbb{R}\liminf_{N\to \infty} 2\beta^2E\left[	\int_0^{u}\int_\mathbb{R}\left(u^{(N)}(r,z)\right)^2	\left(p_{u-r+\ve}(x+z)-p_{u-r+\ve'}(x+z)\right)^2Y_s^2(x)dz dr	\right]dxds\\
&\leq C(T,\delta,X_0,Y_0)|\ve-\ve'|^{\delta}\left((u+\ve)^{-5/8-29\delta/8}+(u+\ve')^{-5/8-29\delta/8}\right)\\
&+C(T,\gamma,\delta,X_0,Y_0)|\ve-\ve'|^{\delta}(u+\ve)^{-\delta}\\
&+C(T,\beta,\delta,X_0,Y_0)|\ve-\ve'|^{\delta}(u+\ve)^{-1/2-\delta}+(u+\ve')^{-1/2-\delta}.
\end{align*}
Therefore, by taking $0<  \delta<\frac{1}{2}$ small enough, we can obtain  that\begin{align*}
\int_0^T\int_0^{T-u}E\left[		\left\langle	\left|\left(X_{u}^\ve\right)^2-X_u^2\right|,Y_s^2	\right\rangle	\right]dsdu\leq C(T,\beta,\gamma,\delta,X_0,Y_0)\ve^{\delta/2}.
\end{align*}
Thus, we have completed the proof.

\end{proof}



\section{Proof of some facts}\label{4}
This section is devoted to the proof of some lemmas used in section \ref{3}.

\begin{lem}\label{lem13}
For any $\beta>0$ and $K>0$, we have that \begin{align}
\sup_{N}E_{Y^1Y^2}\left[		\left(	1+\frac{\beta^2}{N^{\frac{1}{2}}}			\right)^{\sharp \left\{	1\leq i\leq \left\lfloor KN\right\rfloor:Y^1_i=Y^2	_i						\right\}}			\right]<\infty,\notag
\end{align}
where $Y^{1}_{n},$ $Y^{2}_n$ are independent simple random walks on $\mathbb{Z}$.
Also, \begin{align}
&E_{Y^1Y^2}\left[			\left(					1+\frac{\beta^2}{N^{\frac{1}{2}}}			\right)^{\sharp \left\{			1\leq i\leq \left\lfloor KN\right\rfloor:Y^1_{i}=Y^2_i				\right\}}	:Y^1_{\left\lfloor KN\right\rfloor}=x,Y^2_{\left\lfloor KN\right\rfloor}=y				\right]\notag\\
&\hspace{5em}\leq \frac{C}{K^{\frac{1}{2}}N^{\frac{1}{2}}}\left(P_{Y^1}\left(Y^1_{\left\lfloor KN\right\rfloor}=x\right)	\wedge P_{Y^1}\left(		Y^1_{\left\lfloor KN\right\rfloor}=y		\right)		\right).\notag
\end{align}
\end{lem}

\begin{proof}
First, we remark that \begin{align}
&E_{Y^1Y^2}\left[\left(		1+\frac{\beta^2}{N^{\frac{1}{2}}}			\right)^{\sharp \left\{		1\leq i\leq \left\lfloor KN\right\rfloor:Y^1_i=Y^2_i		\right\}}\right]=
	E_{Y^1Y^2}\left[	\prod_{k=1}^{\left\lfloor	KN		\right\rfloor}\left(1+\frac{\beta^2}{N^{\frac{1}{2}}}{\bf 1}\left\{	Y^{1}_{k}=Y^2_{k}		\right\}					\right)			\right]		\notag\\
&=		\sum_{k=0}^{\infty}\frac{\beta^{2k}}{N^{\frac{k}{2}}}\sum_{{\bf i}\in D^k(\left\lfloor KN\right\rfloor)}\sum_{{\bf x}\in \mathbb{Z}^{k}}	
	P_{Y^1Y^2}\left(Y^{1}_{i_{j}}=Y^2_{i_{j}}=x_{j}, \text{ for }1\leq j\leq k	\right)\notag\\
&=		\sum_{k=0}^{\infty}\frac{\beta^{2k}}{N^{\frac{k}{2}}}\sum_{{\bf i}\in D^k(\left\lfloor KN\right\rfloor)}\sum_{{\bf x}\in \mathbb{Z}^{k}}	
	P_{Y}\left(Y_{i_{j}}=x_{j}, \text{ for }1\leq j\leq k	\right)^2		,			\label{beta2}
\end{align}
where $D^k(\left\lfloor KN\right\rfloor)$ is the set defined by \begin{align}
D^k(n)=\left\{{\bf i}=(i_{j})_{j=1}^k\in \mathbb{N}^k:	1\leq i_{1}< \cdots<i_{k}\leq n			\right\},\notag
\end{align}
and the summation for $k> \left\lfloor KN\right\rfloor$ is equal to $0$. By the local limit theorem\begin{align}
&\sum_{{\bf i}\in D^k(\left\lfloor KN\right\rfloor)}\sum_{{\bf x}\in \mathbb{Z}^{k}}	
	P_{Y}\left(Y_{i_{j}}=x_{j}, \text{ for }1\leq j\leq k	\right)^2\notag\\
		&\hspace{3em}\leq C^k\sum_{{\bf i}\in D^k(\left\lfloor KN\right\rfloor)}\sum_{{\bf x}\in \mathbb{Z}^k}\prod_{j=1}^{k}\frac{P_{Y}\left(	Y_{i_{j}-i_{j-1}}=x_{j}-x_{j-1}		\right)}{\sqrt{i_{j}-i_{j-1}}}\notag\\
	&\hspace{3em}\leq		C^k\sum_{{\bf i}\in D^k(\left\lfloor KN\right\rfloor)}\prod_{j=1}^k\frac{1}{\sqrt{i_{j}-i_{j-1}}}.\notag
\end{align}
Thus, we have that \begin{align}
\text{(\ref{beta2})}\leq \sum_{k=0}^{\infty}\frac{\beta^{2k}C^k}{N^k}\sum_{{\bf i}\in D^k(\left\lfloor KN\right\rfloor)}\prod_{j=1}^{k}\frac{1}{\sqrt{\frac{i_{j}}{N}-\frac{i_{j-1}}{N}}}.\label{beta3}
\end{align}

Since $\frac{1}{\sqrt{t-s}}$ is decreasing in $t\in (s,\infty)$, it follows that \begin{align*}
\frac{1}{N^{k}}\prod_{j=1}^{k}\frac{1}{\sqrt{\frac{i_j}{N}-\frac{i_{j-1}}{N}}}\leq \prod_{i=1}^{k}\int_{\frac{i_{j-1}}{N}}^{\frac{i_j}{N}}\frac{dt_j}{\sqrt{t_j-\frac{i_{j-1}}{N}}},
\end{align*}
and \begin{align}
(\ref{beta3})&	\leq \sum_{k=0}^{\infty}{\beta^{2k}C^k}\sum_{{\bf i}\in D^k(\lfloor		KN	\rfloor)}\int_{\frac{i_{k-1}}{N}}^{\frac{i_k}{N}}\cdots \int _{0}^{\frac{i_1}{N}}\prod_{j=1}^k	\left(\frac{1}{\sqrt{t_j-\frac{i_{j-1}}{N}}}\right)d{\bf t}		\notag\\
&\leq \sum_{k=0}^{\infty}\beta^{2k}C^k  \int_{0<t_{1}<\cdots <t_{k}<K}\prod_{j=1}^{k}\frac{1}{\sqrt{t_{j}-t_{j-1}}}d{\bf t}\notag\\
&=\sum_{k=0}^{\infty}\frac{\beta^{2k}{C}^k(\pi K)^{\frac{k}{2}}}{\Gamma(\frac{k}{2}+1)}.\notag
\end{align}
Since $\dis \Gamma\left(\frac{k}{2}+1\right)$ is increase faster than $a^k$ for any $a>1$, the summation is finite for any $\beta$.

Also, the similar argument does hold so that \begin{align}
&E_{Y^1Y^2}\left[\left(		1+\frac{\beta^2}{N^{\frac{1}{2}}}				\right)^{\sharp\left\{		1\leq i\leq \left\lfloor KN\right\rfloor:Y^1_i=Y^2_i			\right\}}:Y^1_{\left\lfloor KN\right\rfloor }=x,Y^2_{\left\lfloor KN\right\rfloor }=y					\right]\notag\\
&=\sum_{k=1}^{\infty}\frac{\beta^{2(k-1)}}{N^{\frac{k-1}{2}}}\sum_{{\bf i}\in D^{k-1}(\left\lfloor KN\right\rfloor-1)}\sum_{{\bf x}\in \mathbb{Z}^{k-1}}\left(			\begin{subarray}{l}	P_{Y}\left(Y_{i_{j}}=x_{j},\ \text{for }1\leq j\leq k-1, Y_{\left\lfloor KN\right\rfloor}=x\right)		\\	\times P_{Y}\left(Y_{i_{j}}=x_{j},\ \text{for }1\leq j\leq k-1, Y_{\left\lfloor KN\right\rfloor}=y\right)\end{subarray}\right)\notag\\
&\ \ +\sum_{k=1}^{\infty}\frac{\beta^{2k}}{N^{\frac{k}{2}}}\sum_{{\bf i}\in D^{k-1}(\lfloor KN\rfloor -1)}\sum_{{\bf x}\in \mathbb{Z}^{k-1}}\left(			\begin{subarray}{l}	P_{Y}\left(Y_{i_{j}}=x_{j},\ \text{for }1\leq j\leq k-1, Y_{\left\lfloor KN\right\rfloor}=x\right)		\\	\times P_{Y}\left(Y_{i_{j}}=x_{j},\ \text{for }1\leq j\leq k-1, Y_{\left\lfloor KN\right\rfloor}=y\right)\end{subarray}\right)\notag\\
&\leq \sum_{k=1}^{\infty}2C^{k}\frac{\beta^{2(k-1)}}{N^{\frac{k-1}{2}}}\sum_{{\bf i}\in D^{k-1}(\lfloor KN\rfloor-1)}\left(\prod_{j=1}^{k-1}\frac{1}{\sqrt{i_{j}-i_{j-1}}}\right)\frac{P_{Y}\left(Y^1_{\lfloor KN\rfloor }=x\right)\wedge P_{Y}\left(		Y^1_{\left\lfloor		KN	\right\rfloor}=y		\right)}{\sqrt{\lfloor KN\rfloor-i_{k-1}}}\notag\\
&\leq \sum_{k=1}^{\infty}\frac{C^k\beta^{2(k-1)}}{	N^{\frac{1}{2}}		}\frac{P_{Y}\left(Y^1_{\lfloor KN\rfloor}=x\right)\wedge P_{Y}\left(		Y^1_{\left\lfloor		KN	\right\rfloor}=y		\right)}{N^{k-1}}\sum_{{\bf i}\in D^{k-1}(\lfloor KN\rfloor-1)}\prod_{j-1}^{k-1}\frac{1}{\sqrt{\frac{i_{j}}{N}-\frac{i_{j-1}}{N}}}\frac{1}{\sqrt{K-\frac{i_{k-1}}{N}}}.\label{endpoint}
\end{align}
By the integration by parts, we have that \begin{align*}
\int_{\frac{i_{k-2}}{N}}^{\frac{i_{k-1}}{N}}\frac{1}{\sqrt{t_{k-1}-\frac{i_{k-2}}{N}}\sqrt{K-t_{k-1}}}dt_{k-1}&=\left[	2\frac{\sqrt{t_{k-1}-\frac{i_{k-2}}{N}}}{\sqrt{K-t_{k-1}}}			\right]_{\frac{i_{k-2}}{N}}^{\frac{i_{k-1}}{N}}+\text{positive term}\\
&\geq 2\frac{\sqrt{\frac{i_{k-1}}{N}-\frac{i_{k-2}}{N}}}{\sqrt{K-\frac{i_{k-1}}{N}}} \\
&\geq \frac{2}{N}\frac{1}{\sqrt{\frac{i_{k-1}}{N}-\frac{i_{k-2}}{N}}\sqrt{K-\frac{i_{k-1}}{N}}}.
\end{align*}
Also, we know that \begin{align*}
&\sum_{{\bf i}\in D^{k-1}(\lfloor KN\rfloor-1)}\left(\prod_{j=1}^{k-2}	\int_{\frac{i_{j-1}}{N}}^{\frac{i_{j}}{N}}\frac{1}{\sqrt{t_{j}-\frac{i_{j}}{N}}}dt_{j}\right)		\left(\int_{\frac{i_{k-2}}{N}}^{\frac{i_{k-1}}{N}}\frac{1}{\sqrt{t_{k-1}-\frac{i_{k-2}}{N}}\sqrt{K-t_{k-1}}}dt_{k-1}\right)\\
&\leq \int_{0<t_1<\cdots<t_{k-1}<K}\prod_{j=1}^{k-1}\left(\frac{1}{\sqrt{t_j-t_{j-1}}}\right)\frac{1}{\sqrt{K-t_{k-1}}}d{\bf t}\\
&\leq\frac{\pi^{{\frac{k}{2}}} K^{\frac{k-1}{2}}}{K^{\frac{1}{2}}\Gamma\left(	\frac{k-1}{2}	\right)}.
\end{align*}
Thus, we have that \begin{align*}
&\text{(\ref{endpoint})}\\
&\leq \frac{P_{Y}\left(Y^1_{\lfloor KN\rfloor}=x\right)\wedge P_{Y}\left(		Y^1_{\left\lfloor		KN	\right\rfloor}=y		\right)}{(KN)^{\frac{1}{2}}}\sum_{k=1}^{\infty}\frac{C^k\beta^{2(k-1)}{K^{\frac{k-1}{2}}}}{\Gamma\left(	\frac{k-1}{2}		\right)}.
\end{align*}
Since  the  summation is finite for any $\beta\in \mathbb{R}$, the statement holds.

\end{proof}

The next lemma gives us an upper bound of $p$-th moment of $B_n$ for branching random walks in random environment.

\begin{lem}\label{lemma14}
If $E[m^{(p)}_{n,x}]=K<\infty$ for $p\in \mathbb{N}$ and $E\left[m_{n,x}^{(1)}\right]=1$, then \begin{align} 
E\left[	B^{p}_{n}		\right]		&\leq 		C(p,K)n^{p-1}E_{Y^1\cdots Y^p}\left[E\left[\left(m_{0,0}^{(1)}\right)^{p}\right]^{\sharp\left\{1\leq i\leq n:Y^a_i=Y^b_i,a\not=b\in \{1,\cdots,p\}						\right\}}		\right]\notag
\intertext{and}
E\left[	\prod_{i=1}^{p}B_{n,x_i}						\right]&\notag\\
&\hspace{-3em}\leq C(p,K)n^{p-1}E_{Y^1\cdots Y^p}\left[			E\left[		\left(m_{0,0}^{(1)}\right)^p			\right]	^{\sharp\left\{1\leq i\leq n:Y^a_i=Y^b_i,a\not=b\in \{1,\cdots,p\}						\right\}}: Y^{i}_{n}=x_{i}\text{ for }1\leq i\leq p	\right].\notag
\end{align}
\end{lem}

Before starting a proof, we give another representation of $B_{n}$. Let $\{V^{\x}_{n,x}:\x\in {\cal T},(n,x)\in \mathbb{N}\times \mathbb{Z}^d\}$ be $\mathbb{N}$-valued random variables with $P\left(\left.	V^{\x}_{n,x}=k		\right|\omega\right)=q_{n,x}(k)$. Let $\{	X^{\x}_{n,x}:\x\in {\cal T},(n,x)\in \mathbb{N}\times \mathbb{Z}^d		\}$ be i.i.d.\,random variables with $P(X^{\x}_{n,x}=e)=\frac{1}{2d}$ for $e=\pm e_{j}$, $j=1,\cdots,d$ where $e_{j}$ are unit vector on $\mathbb{Z}^d$. $V_{n,x}^{\x}$ denotes the number of offsprings of $\x$ if $\x$ locates at $x$ at time $n$ and $X_{n,x}^{\x}$ denotes the step of $\x$ if it locates at $x$ at time $n$.

We consider the event $\left\{	\text{particle $\y$ exists and locates at site $y$ at time $|\x|=n$}					\right\}$ and its indicator function \begin{align}
B_{n,y}^{\y}={\bf 1}\left\{		\text{particle $\y$ exists and locates at site $y$ at time $|\x|=n$}					\right\}\notag.
\end{align}
Then, it is clear that \begin{align}
B^{{ \x}}_{0,x}&=\delta_{x,\x}=\begin{cases}
1\ \ \text{if $x=0$ and $\x={\bf 1}$},\\
0\ \ \text{otherwise},
\end{cases}				\notag\\
B_{n,y}^{\y}&=\sum_{x,\x}B^{\x}_{n-1,x}{\bf 1}\left\{	X_{n-1,x}^{\x}=y-x,V^{\y}_{n-1,x}\geq \y/\x\geq 1				\right\}\notag\\
&=\sum_{0\to y}\sum_{{\bf 1}\to \y}\prod_{i=0}^{n-1}{\bf 1}\left\{		X_{i,y_{i}}^{\y_i}=y_{i+1}-y_i,V^{\y_{i}}_{i,y_{i}}\geq \y_{i+1}/\y_{i}\geq 1				\right\},\notag\\
\intertext{and}
&B_{n,y}=\sum_{\y}\sum_{0\to y}\sum_{{\bf 1}\to \y}\prod_{i=0}^{n-1}{\bf 1}\left\{		X_{i,y_{i}}^{\y_i}=y_{i+1}-y_i,V^{\y_{i}}_{i,y_{i}}\geq \y_{i+1}/\y_{i}\geq 1				\right\}.\notag
\end{align}
We introduce new Markov chain ${\bf Y}=(Y,\mathbb{Y})$ on $\mathbb{Z}^d\times {\cal T}$ which are determined by \begin{align}
&Y_{0}=0,\Y_{0}={\bf 1}\in T_{0}.				\notag\\
&P_{Y\Y}\left(	\left.		\begin{subarray}{c}	Y_{n+1}=y,\\ \Y_{n+1}=\y		\end{subarray}\right|\begin{subarray}{c}Y_{n}=x,\\ \Y_{n}=\x\end{subarray}	\right)=\begin{cases}
\frac{1}{2d}\sum_{k\geq \y/\x}q({k})\ \ &\text{if }|y-x|=1,\y/\x<\infty,\notag\\
0&\text{otherwise},
\end{cases}
\end{align}
where $q(k)=E[q_{n,x}(k)]$. Let $A_{n,x,y}^{\x,\y}={\bf 1}\left\{X_{n,x}^{\x}=y-x,V^{\x}\geq \y/\x\right\}$. Then, we have the following representation of $B_{n,y}$ \cite{MN}: \begin{align}
&B_{n,y}=E_{Y\Y}\left[\prod_{i=0}^{n-1}\frac{A_{i,Y_{i},Y_{i+1}}^{\Y_{i},\Y_{i+1}}	}{E\left[A_{i,Y_{i},Y_{i+1}}^{\Y_{i},\Y_{i+1}}\right]	}	:Y_{n}=y					\right],\notag
\end{align}
{and also} \begin{align}
&E\left[	\prod_{i=1}^pB_{n,x_i}			\right]=E_{{\bf Y}^1\cdots {\bf Y}^{p}}\left[		\prod_{i=0}^{n-1}E\left[	\frac{\prod_{j=1}^{p}	A^{\Y^j_{i},\Y^{j}_{i+1}}_{i,Y^j_{i},Y^j_{i+1}}			}{\prod_{j=1}^{p}E\left[	A^{\Y^j_{i},\Y^{j+1}_{i+1}}_{i,Y^j_{i},Y^j_{i+1}}				\right]}						\right]		:Y^i_n=x_i\text{ for }1\leq i\leq p						\right]\notag\\
&E\left[	B_{n}^{p}			\right]=E_{{\bf Y}^1\cdots {\bf Y}^{p}}\left[		\prod_{i=0}^{n-1}E\left[	\frac{\prod_{j=1}^{p}	A^{\Y^j_{i},\Y^{j}_{i+1}}_{i,Y^j_{i},Y^j_{i+1}}			}{\prod_{j=1}^{p}E\left[	A^{\Y^j_{i},\Y^{j+1}_{i+1}}_{i,Y^j_{i},Y^j_{i+1}}				\right]}						\right]								\right],\notag
\end{align}
where ${\bf Y}^i=(Y^i,\Y^i)$ are independent copies of ${\bf Y}=(Y,\Y)$.

\begin{proof}[Proof of Lemma \ref{lemma14}]
We remark the following facts:\begin{enumerate}[i)]
\item If $y\not=y'$, then $A^{\x,\y}_{i,x,y}A^{\x,\y'}_{i,x,y'}=0$ almost surely. Especially, for $\left\{{\bf Y}^{j}_{i}:i=0,\cdots, n\right\}$ and $\left\{{\bf Y}^{j'}_{i}:i=0,\cdots,n		\right\}$, if there exists an $i$ such that ${\bf Y}^{j}_{i}={\bf Y}^{j'}_{i}$ and $Y^{j}_{i+1}\not=Y^{j'}_{i+1}$, then\begin{align} 		\prod_{i=0}^{n-1}E\left[	\frac{\prod_{j=1}^{p}	A^{\Y^j_{i},\Y^{j}_{i+1}}_{i,Y^j_{i},Y^j_{i+1}}			}{\prod_{j=1}^{p}E\left[	A^{\Y^j_i,\Y^{j+1}_{i+1}}_{i,Y^j_{i},Y^j_{i+1}}				\right]}						\right]	=0,	\notag
	\end{align}
	almost surely.
\item If $\y/\x=k$, $\y'/\x=\ell$, and $k\leq \ell$, then $	A^{\x,\y}_{i,x,y}A^{\x,\y'}_{i,x,y}			=A^{\x,\y}_{i,x,y}$ almost surely.
\item If $\{\x^{j}:j=1,\cdots ,p\}$ are different from each other and $\y^{j}/\x^{j}=k_{j}$, then $E\left[		\prod_{j=1}^rA^{\x^{j},\y^{j}}_{i,x^j,y^j}		\right]=\left(\frac{1}{2d}\right)^p		\sum_{s_{1}\geq k_{1}}\cdots \sum_{s_{p}\geq k_{p}}E\left[	\prod_{j=1}^{p}{q_{i,x^j}}(s_{j})			\right]	$.
\end{enumerate} 
Thus,  the possible cases are the followings:\begin{align}
&E\left[E_{{\bf Y}^{1}\cdots {\bf Y}^{p}}\left[	\left.					\frac{\prod_{j=1}^{p}	A^{\Y^j_{i},\Y^{j}_{i+1}}_{i,Y^j_{i},Y^j_{i+1}}			}{\prod_{j=1}^{p}E\left[	A^{\Y^j_{i},\Y^{j}_{i+1}}_{i,Y^j_{i},Y^j_{i+1}}				\right]}									\right|	Y_{i}^{j}=x^j,\mathbb{Y}^{j}_{i}=\x^{j}	 \text{ for }j=1,\cdots p	\right]\right]\notag\\
&=\begin{cases}\notag
1						\ \ \ \ \text{$x^{j}$ are different from each others,}\\
{E\left[	\prod_{j=1}^{p}m_{i,y^j}^{(1)}			\right]}		\ \ \ \ \text{if $\x^{j}$ are different from each others,}\\
(A),
\end{cases}
\end{align}
where $(A)$ is the other case described as below. 

We divide the set $\{1,\cdots, p\}$ into the disjoint union such that \begin{align}
\{1,\cdots,p\}=\coprod_{k=j_1}^{j_{p}}I_{k},\label{eqcl}
\end{align}
where $I_{k}=\{		j\in\{1,\cdots,p\}:\x^j=\x^k		\}$ and $j_1,\cdots j_{p}$ is the set of index of equivalence class $I_{k}$. For $\y^j/\x^j=k_j$, we set $K_{j_{\ell}}=\min\{k_j:j\in I_{j_{\ell}}		\}$.
Then, we have that \begin{align}
&E\left[E_{{\bf Y}^{1}\cdots {\bf Y}^{p}}\left[	\left.					\frac{\prod_{j=1}^{p}	A^{\Y^j_{i},\Y^{j}_{i+1}}_{i,Y^j_{i},Y^j_{i+1}}			}{\prod_{j=1}^{p}E\left[	A^{\Y^j_{i},\Y^{j}_{i+1}}_{i,Y^j_{i},Y^j_{i+1}}				\right]}					{\bf 1}\left\{	\Y^{j}_{i+1}=\y^{j}\ \text{for }j=1,\cdots,p			\right\}				\right|	Y_{i}^{j}=x^j,\mathbb{Y}^{j}_{i}=\x^{j}	 \text{ for }j=1,\cdots, p	\right]\right]\notag\\
&\hspace{10em}=E\left[		\prod_{\ell=j_{1}}^{j_{p}}\left(\sum_{k\geq K_{\ell}}q_{i,x^{\ell}}(k)\right)				\right].\notag
\end{align}
By the above argument, we find that ${\bf Y}^{1},\cdots,{\bf Y}^{P}$ evolves according the following steps:
\begin{enumerate}[i)]
\item First, the set process $\{S(m):m=0,\cdots, n\}$ starts from the set $I^{(0)}=\{1,\cdots,p\}$ until time $i^{(1)}$, and then it splits into some sets $I^{(1,1)},\cdots,I^{(1,k^{(1)})}$. ($i^{(1)}$ is the last  time when ${\bf Y}^{j}_{i}$ coincide and $I^{(1,1)},\cdots, I^{(1,k^{(1)})}$ are the equivalent class defined in (\ref{eqcl}) for $\Y^{j}_{i^{(1)}+1}$).
\item When the set process $S(m)=\{I^{(\ell,1)},\cdots, I^{(\ell,k^{(\ell)})}\}$, it jumps to the new sets $\{I^{(\ell+1,1)},\cdots, I^{(\ell+1,k^{(\ell+1)})}\}$ where each $I^{(\ell+1,r)}$ is a partition of some set of $I^{(\ell,1)},\cdots I^{(\ell,k^{(\ell)})}$ at some time $i^{(\ell+1)}$. (${\bf Y}^{(j)}$, $j\in I^{(\ell,s)}$ for each $s=1,\cdots,k^{(\ell)}$ coincides until time $i^{(\ell+1)}$ and $\mathbb{Y}^{j}_{i^{(\ell+1)}+1}\not=\mathbb{Y}^{j'}_{i^{(\ell+1)}+1}$ for some $j,j'\in I^{(\ell,k)}$ for some $k$).
\item If $S(m)=\{\{1\},\cdots,\{p\}\}$, then $S(m)=S(m')$ for $m'\geq m$.
\end{enumerate}

\begin{figure}[h]
\begin{center}
\includegraphics[width=3in,height=2in]{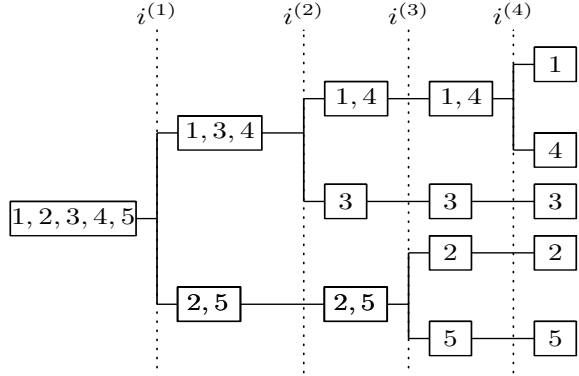}
\caption{When $p=5$, $I^{(0)}=\{1,2,3,4,5\}$.	In this figure, $I^{(1,1)}=\{1,3,4\}$, $I^{(1,2)}=\{2,5\}$, $I^{(2,1)}=\{1,4\}$, $I^{(2,2)}=\{3\}$, and $I^{(2,3)}=\{2,5\}$.  			}
\end{center}
\end{figure}

First, we remark that the combination of $i^{(1)},\cdots, i^{(p-1)}$ (it may stops for less steps) are at most $n^{p}$-th order. Also, \begin{align}
&E\left[E_{{\bf Y}^{1}\cdots {\bf Y}^{p},S}\left[	\left.					\frac{\prod_{j=1}^{p}	A^{\Y^j_{i},\Y^{j}_{i+1}}_{i,Y^j_{i},Y^j_{i+1}}			}{\prod_{j=1}^{p}E\left[	A^{\Y^j_{i},\Y^{j}_{i+1}}_{i,Y^j_{i},Y^j_{i+1}}				\right]}					{\bf 1}\left\{	i^{(\ell)}=i		\right\}				\right|	Y_{i}^{j}=x^j,\mathbb{Y}^{j}_{i}=\x^{j}	 \text{ for }j=1,\cdots, p	\right]\right]\notag\\
&\leq		C(p)K,\notag
\intertext{and }
&E\left[E_{{\bf Y}^{1}\cdots {\bf Y}^{p},S}\left[	\left.					\frac{\prod_{j=1}^{p}	A^{\Y^j_{i},\Y^{j}_{i+1}}_{i,Y^j_{i},Y^j_{i+1}}			}{\prod_{j=1}^{p}E\left[	A^{\Y^j_{i},\Y^{j}_{i+1}}_{i,Y^j_{i},Y^j_{i+1}}				\right]}					{\bf 1}\left\{	i^{(\ell)}\not=i,\text{ for }\ell=1,\cdots,p		\right\}				\right|	Y_{i}^{j}=x^j,\mathbb{Y}^{j}_{i}=\x^{j}	 \text{ for }j=1,\cdots, p	\right]\right]\notag\\
&\leq		\prod_{k\in {\cal K}}E\left[	\left(m_{i,x^{k}}\right)^{\sharp\left\{j:x^j=x^k				\right\}}			\right]\leq \prod_{k\in{\cal K}}E\left[		(m_{i,x^k})^p		\right]^{\sharp\left\{	j:x_j=x_k			\right\}/p}\leq E	\left[		(m_{i,x^k})^p		\right]^{{\bf 1}\left\{	x^j=x^k, \text{ for some }j\not=k		\right\}}	,\notag
\end{align}
where ${\cal K}$ be the set of index for equivalence class $\left\{j:x^j=x^k				\right\}$.

Thus, we have that \begin{align}
E\left[B_{n}^P\right]\leq C(p,K)n^{(p-1)}E_{{\bf Y}^1\cdots {\bf Y}^p}\left[	
	E\left[	\left(m_{n,x}\right)^{p}				\right]^{\sharp\left\{i\leq n	:	Y^j_i=Y^{j'}_i\ \text{for }j\not=j'\in \{1,\cdots ,p\}			\right\}}				
		\right]		\notag.
\end{align}
The latter part of Lemma \ref{lemma14} can be proved by the same argument.

\end{proof}

\begin{cor}\label{cor1}
Under the same assumption in Lemma \ref{lemma14}, \begin{align}
&E\left[	\prod_{j=1}^{q}\prod_{i=1}^{p_j}B^{(j)}_{n,x_{(j,i)}}			\right]\notag\\
&\leq C({\bf p},K)n^{\left(\sum_{j=1}^{q}p_j-q\right)}E_{\left({\bf Y}^{j,i}\right)}\left[	E\left[\left(	m_{0,0}\right)^{\sum_{j=1}^{q}p_j}	\right]	^{\sharp\left\{\begin{subarray}{l}	k\leq n:Y^{j_1,i_1}_{k}=Y^{j_2,i_2}_{k},\ \text{for }\\ (j_1,i_1)\not=(j_2,i_2)\in \left\{(j,i):j=1,\cdots, q,		i=1,\cdots,p_j\right\}	\end{subarray}	\right\}}	:Y^{(j,i)}_n=x_{j,i}		\right]\notag,
\end{align}
where $B^{(j)}_{n,x}$ is the number of particles from initial particle $j$ at site $x$ at time $n$. 
\end{cor}
\begin{proof}
If we regard $i^{(1)}=-1$ and $S(0)=\{\{1,\cdots, p_1\},\cdots,\{\sum_{j=1}^{q-1}p_j+1,\cdots,\sum_{j=1}^{q}p_j\}			\}$, then $S(m)$ stops at $\{\{1\},\cdots,\{\sum_{j=1}^qp_j\}\}$ at most $\sum_{j=1}^{q}p_j-q$ jumps.
\end{proof}


\end{document}